\documentclass[11pt]{amsart}
\usepackage[margin=1in]{geometry}

\usepackage{amssymb}
\usepackage{amsthm}
\usepackage{amsmath}
\usepackage{mathrsfs}
\usepackage{amsbsy}
\usepackage{bm}
\usepackage{hyperref}
\usepackage{tikz}
\usepackage{array}
\usepackage{enumerate}
\usepackage{xcolor}
\usepackage{bbm}
\usepackage{comment}
\usepackage{mathtools}

\definecolor{Green}{rgb}{0,0.922,0}
\definecolor{DarkGreen}{rgb}{0,0.5,0}
\definecolor{MildGreen}{rgb}{0,0.784,0}
\definecolor{NormalGreen}{rgb}{0,0.8,0}
\definecolor{LightGreen}{rgb}{0,0.922,0}
\definecolor{Magenta}{rgb}{1,0,0.6}
\definecolor{Yellow}{rgb}{0.95,0.95,0}
\definecolor{Gold}{rgb}{1,0.73,0}

\usepackage[noabbrev,capitalise,nameinlink]{cleveref}

\hypersetup{colorlinks=true, citecolor=NormalGreen, linkcolor=DarkGreen,urlcolor=DarkGreen}

\usepackage{hhline}
\allowdisplaybreaks
\usepackage[noadjust]{cite}

\usepackage{caption}
\usepackage[noabbrev,capitalise,nameinlink]{cleveref}
\crefname{conjecture}{Conjecture}{Conjectures}

\newtheorem{theorem}{Theorem}[section]
\newtheorem{proposition}[theorem]{Proposition}
\newtheorem{corollary}[theorem]{Corollary}

\newtheorem{lemma}[theorem]{Lemma}

\theoremstyle{definition}
\newtheorem{definition}[theorem]{Definition}
\newtheorem{remark}[theorem]{Remark}
\newtheorem{example}[theorem]{Example}

\usepackage{etoolbox}

\newcommand{\includeSymbol}[1]{\ensuremath{%
	\mathchoice
		{\raisebox{-.7mm}{\includegraphics[height=2.2ex]{#1}}}	
		{\raisebox{-.7mm}{\includegraphics[height=2.2ex]{#1}}}
		{\raisebox{-.6mm}{\includegraphics[height=1.6ex]{#1}}}
		{\raisebox{-.5mm}{\includegraphics[height=1ex]{#1}}}
}}

\robustify{\includeSymbol}

\newcommand{\ob}{\mathrm{ob}} 
\newcommand{\dH}{\mathrm{d}_{\mathrm{H}}}
\newcommand{\ADD}{\partial^*}
\newcommand{\proj}{\mathrm{proj}}

\newcommand{\wt}{\mathrm{wt}}
\newcommand{\TT}{\mathrm{T}}
\newcommand{\MM}{\mathfrak{K}}
\newcommand{\QQ}{\mathsf{Q}}
\newcommand{\dd}{\eta}
\newcommand{\pp}{\mathsf{p}}
\newcommand{\bp}{\mathsf{bp}}
\newcommand{\ddd}{\delta}
\newcommand{\rr}{\mathfrak{r}}

\newcommand{\Omm}{\Omega}
\newcommand{\ttt}{\boldsymbol{\chi}}
\newcommand{\wo}{w_{\circ}}
\newcommand{\PP}{\mathbb{P}}
\newcommand{\affS}{\widetilde{\mathfrak{S}}}
\newcommand{\HH}{\mathrm{H}}
\newcommand{\BB}{\mathcal{A}}
\newcommand{\CC}{\mathcal{C}}

\newcommand{\idd}{\mathbbm{1}}
\newcommand{\ii}{\mathrm{i}}
\newcommand{\hh}{\mathfrak{f}}
\newcommand{\s}{\blacktriangle}
\newcommand{\SSS}{\mathfrak{S}}
\newcommand{\AmASEP}{\mathrm{ASEP}}
\newcommand{\iTASEP}{\mathrm{iTASEP}}
\newcommand{\CmASEP}{\mathrm{ASEP}^{\mathrm{ob}}}
\newcommand{\Lam}{\mathrm{Lam}}

\newcommand{\ZZ}{\mathbb{Z}}

\newcommand{\qqq}{\mathfrak{q}}
\newcommand{\rrho}{\varrho}

\newcommand{\dfn}[1]{\textcolor{NormalGreen}{\emph{#1}}}

\begin{document}

\title[]{Random Combinatorial Billiards and \\ Stoned Exclusion Processes}
\subjclass[2010]{}

\author[]{Colin Defant}
\address[]{Department of Mathematics, Harvard University, Cambridge, MA 02138, USA}
\email{colindefant@gmail.com}

\begin{abstract}
We introduce and study several \emph{random combinatorial billiard trajectories}. Such a system, which depends on a fixed parameter $p\in(0,1)$, models a beam of light that travels in a Euclidean space, occasionally randomly reflecting off of a hyperplane in the Coxeter arrangement of an affine Weyl group with some probability that depends on the side of the hyperplane that it hits. In one case, we recover Lam's reduced random walk in the limit as $p$ tends to $0$. The investigation of our random billiard trajectories relies on an analysis of new finite Markov chains that we call \emph{stoned exclusion processes}. These processes have remarkable stationary distributions determined by well-studied polynomials such as ASEP polynomials, inhomogeneous TASEP polynomials, and open boundary ASEP polynomials; in many cases, it was previously not known how to construct Markov chains with these stationary distributions. Using multiline queues, we analyze correlations in the \emph{stoned multispecies TASEP}, allowing us to determine limit directions for reduced random billiard trajectories and limit shapes for new random growth processes for $n$-core partitions. 
\end{abstract} 

\maketitle

\section{Introduction}\label{sec:intro}

\subsection{Weyl Groups}\label{subsec:IntroWeyl}  
Let $\Phi$ be a finite irreducible crystallographic root system spanning a Euclidean space $V$, and write $\Phi=\Phi^+\sqcup\Phi^-$, where $\Phi^+$ and $\Phi^-=-\Phi^+$ are the set of positive roots and the set of negative roots, respectively.
Let $W$ and $\widetilde W$ be the Weyl group and affine Weyl group of $\Phi$, respectively. Let $I$ be an index set so that $\{\alpha_i:i\in I\}$ is the set of simple roots, and let $\widetilde I=\{0\}\sqcup I$. Write $S=\{s_i:i\in I\}$ and $\widetilde S=\{s_i:i\in\widetilde I\}$ for the sets of simple reflections of $W$ and $\widetilde W$, respectively. Let $\theta\in\Phi$ be the highest root of $W$. 

Let $V^*$ be the dual space of $V$. Each root $\beta\in \Phi$ has an associated coroot $\beta^\vee\in V^*$. Let $Q^\vee=\mathrm{span}_{\ZZ}\{\beta^\vee:\beta\in\Phi\}\subseteq V^*$ denote the coroot lattice of $W$.  
For $\beta\in\Phi^+$ and $k\in\ZZ$, we define the hyperplane 
\[\HH_\beta^k=\{\gamma\in V^*:\gamma(\beta)=k\}\subseteq V^*.\]  The \dfn{Coxeter arrangements} of $W$ and $\widetilde W$ are \[\mathcal H_W=\{\HH_{\beta}^0:\beta\in\Phi^+\}\quad\text{and}\quad{\mathcal H}_{\widetilde W}=\{\HH_{\beta}^k:\beta\in\Phi^+,\,\, k\in\ZZ\},\] respectively. 

There is a faithful right action of $\widetilde W$ on $V^*$; each simple reflection $s_i\in S$ acts via the reflection through the hyperplane $\HH_{\alpha_{i}}^0$, while $s_0$ acts via the reflection through $\HH_{\theta}^1$. The closures of the connected components of $V^*\setminus\bigcup_{\HH\in\mathcal H_W}\HH$ are called \dfn{chambers}, while the closures of the connected components of $V^*\setminus\bigcup_{\HH\in\mathcal H_{\widetilde W}}\HH$ are called \dfn{alcoves}. The \dfn{fundamental chamber} is 
\[\CC=\{\gamma\in V^*:\gamma(\alpha_i)\geq 0\text{ for all }i\in I\},\] and the \dfn{fundamental alcove} is 
\[\BB=\{\gamma\in\CC:\gamma(\theta)\leq 1\}.\] 
The map $u\mapsto\CC u$ is a bijection from $W$ to the set of chambers. The map $u\mapsto\BB u$ is a bijection from $\widetilde W$ to the set of alcoves. Two distinct alcoves are \dfn{adjacent} if they share a common facet. The alcoves adjacent to $\BB u$ are precisely the alcoves of the form $\BB su$ for $s\in \widetilde S$. Let $\HH^{(u,s)}$ denote the unique hyperplane separating $\BB u$ and $\BB s u$. 

Consider the $|I|$-dimensional torus $\mathbb T=V^*/Q^\vee$, and let $\qqq\colon V^*\to\mathbb T$ be the natural quotient map. There is a quotient map $\widetilde W\to W$, which we denote by $w\mapsto\overline w$, where $\overline w$ is the unique element of $w$ such that $\qqq(\BB w)=\qqq(\BB\overline w)$. 

\subsection{Reduced Random Walks} 
In \cite{Lam}, Lam introduced the \dfn{reduced random walk} in $\widetilde W$, a very natural and intriguing random walk on the set of alcoves of ${\mathcal H}_{\widetilde W}$ (equivalently, on $\widetilde W$). The reduced random walk starts at $\BB$. Suppose that at some point in time, the walk is at an alcove $\BB u$. A simple reflection $s$ is chosen uniformly at random from $\widetilde S$. If the walk has already crossed through the hyperplane $\HH^{(u, s)}$ (i.e., if $\HH^{(u,s)}$ separates $\BB u$ from $\BB$), then the walk stays at the alcove $\BB u$; otherwise, it transitions to $\BB su$. Let $\widetilde{\bf M}_{\Lam}$ denote the reduced random walk in $\widetilde W$. 

Let \[\widehat W=\{w\in \widetilde W:\BB w\subseteq\CC\}\] denote the set of \dfn{affine Grassmannian} elements of $\widetilde W$. Lam also introduced the \dfn{affine Grassmannian reduced random walk} in $\widehat W$, which is the random walk $\widehat{\bf M}_\Lam$ obtained by conditioning $\widetilde{\bf M}_{\Lam}$ to stay within $\mathcal C$. By projecting $\widehat{\bf M}_\Lam$ through the natural quotient map $\widetilde W\to W$, Lam obtained an irreducible finite-state Markov chain ${\bf M}_{\Lam}$ on $W$; when $W$ is of type~$A_{n-1}$, it turns out that ${\bf M}_{\Lam}$ is isomorphic to the \emph{$n$-species totally asymmetric simple exclusion process} ($n$-species TASEP) on a ring with $n$ sites. The Markov chain ${\bf M}_\Lam$ can also be seen as a certain random walk on the \emph{toric alcoves} in the torus $\mathbb T$. Let $\zeta_\Lam$ denote the stationary probability distribution of ${\bf M}_\Lam$.  

Lam found that the asymptotic behavior of the reduced random walk in $\widetilde W$ is governed by $\zeta_\Lam$; the following theorem summarizes his main results in this vein. For $\gamma\in V^*\setminus\{0\}$, let $\langle \gamma\rangle$ denote the unit vector in $V^*$ that points in the same direction as $\gamma$. By a slight abuse of notation, we write $\langle\BB u\rangle$ for the unit vector that points in the same direction as the center of $\BB u$. Let $\wo$ denote the long element of $W$. 

\begin{theorem}[{Lam~\cite{Lam}}]\label{thm:Lam} 
Let $\BB u_M$ and $\BB v_M$ denote the states at time $M$ of the reduced random walk in $\widetilde W$ and the affine Grassmannian reduced random walk in $\widehat W$, respectively. Let 
\[\psi_\Lam=\sum_{\substack{w\in W\\ w^{-1}\theta\in\Phi^+}}\zeta_\Lam(w)\theta^\vee w.\] With probability $1$, \[\lim_{M\to\infty}\langle \BB v_M\rangle=\langle\psi_{\Lam}\rangle,\] and the limit \[\lim_{M\to\infty}\langle \BB u_M\rangle\] exists and belongs to $\langle\psi_\Lam\rangle W$. Moreover, 
for each $v\in W$, we have \[\PP\left(\lim_{M\to\infty}\langle \BB u_M\rangle=\langle\psi_\Lam\rangle v\right)=\zeta_{\Lam}(v^{-1}\wo).\] 
\end{theorem}

\subsection{Random Billiards}\label{subsec:intro_billiards}
\emph{Dynamical algebraic combinatorics} is a field that studies dynamical systems on objects of interest in algebraic combinatorics (see, e.g., \cite{DefantToric,Haiman,HopkinsRubey,ProppRoby,Rhoades,StanleyPromotion,StrikerSurvey,StrikerWilliams,ThomasWilliams}). \emph{Mathematical billiards} is a subfield of dynamics concerning the trajectory of a beam of light that moves in a straight line except for occasional reflections \cite{Burago, Masur, McMullen,  Tabachnikov}. \emph{Combinatorial billiards} combines these two areas, focusing on mathematical billiard systems that are in some sense rigid and discretized; these billiard systems can usually be modeled combinatorially or algebraically \cite{DefantRefractions, DefantJiradilok, DJM, Zhu, Barkley, BDHKL}. Our first goal in this article is to define a random billiard trajectory that resembles Lam's reduced random walk and its variants. 

\begin{figure}[ht]
\begin{center}{\includegraphics[width=\linewidth]{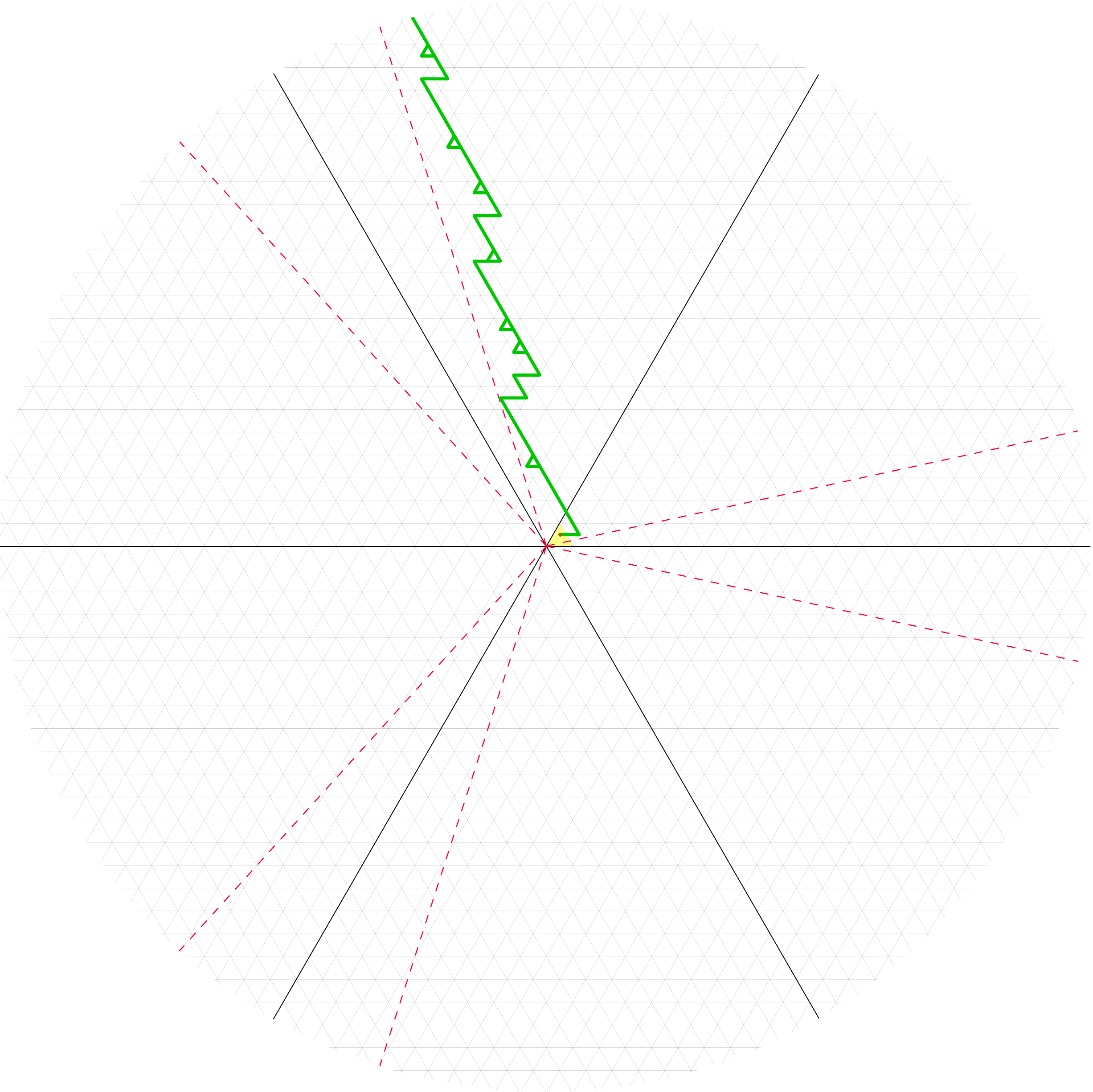}}
\end{center}
\caption{A reduced random billiard trajectory in the affine symmetric group $\affS_3=\widetilde A_2$ with parameter $p=3/4$. The ({\color{MildGreen}green}) beam of light starts in the ({\color{Yellow}yellow}) fundamental alcove traveling in the direction of the vector $\ddd^{(3)}=(1,1,-2)$. Occasionally, the beam of light traverses a small triangle numerous times; however, the number of times it traverses the small triangle is not discernible from the figure. The six possible asymptotic directions of the beam of light are represented by {\color{red}red} dotted rays. The six thick black lines are the hyperplanes in the Coxeter arrangement of the finite Weyl group $\SSS_3=A_2$; they separate the space into six chambers.}   
\label{fig:Trajectory}
\end{figure} 

Fix a point $z_0$ in the interior of $\BB$. For $\dd\in V^*\setminus\{0\}$, let $\mathfrak r_\dd$ be the ray that starts at $z_0$ and travels in the direction of $\dd$. Let $\Upsilon_{z_0}$ denote the set of vectors $\dd\in Q^\vee\setminus\{0\}$ such that $\rr_\dd$ does not pass through the intersection of two or more hyperplanes in $\mathcal H_{\widetilde W}$.

Given $\dd\in\Upsilon_{z_0}$, we can record the sequence $\BB u_0, \BB u_1,\ldots$ of alcoves through which $\mathfrak r_\dd$ passes (in particular, $\BB u_0=\BB$); we then define the infinite word $\mathsf{w}(\dd)=\mathsf{w}_{z_0}(\dd)=\cdots s_{i_1}s_{i_0}$, where $s_{i_j}$ is the unique simple reflection such that $u_{j+1}=s_{i_j}u_{j}$ (our convention is that infinite words extend infinitely to the left). The word $\mathsf{w}(\dd)$ is necessarily periodic since $\dd\in Q^\vee$ (the labeling of the facets of alcoves by simple reflections is invariant under translation by a coroot vector). We let $N=N_\dd$ denote the period of $\mathsf{w}(\dd)$. (See \cref{subsec:introA} for an example.) 

\begin{definition}\label{def:RRBT}
Fix $\dd\in\Upsilon_{z_0}$ and $p\in(0,1)$. Shine a beam of light from $z_0$ in the direction of $\dd$. Whenever the beam of light hits a hyperplane in ${\mathcal H}_{\widetilde W}$ that it has not yet crossed, it passes through the hyperplane with probability $p$ and reflects off of the hyperplane with probability $1-p$. Whenever the beam of light hits a hyperplane in ${\mathcal H}_{\widetilde W}$ that it has already crossed, it reflects off of the hyperplane. We call this random process a \dfn{reduced random billiard trajectory}, and we denote it by $\mathrm{RRBT}_{z_0}(\dd)$. By imposing the extra condition that the beam of light always reflects when it hits a wall of the fundamental chamber, we obtain a different random process that we call the \dfn{affine Grassmannian reduced random billiard trajectory} and denote by $\mathrm{AGRRBT}_{z_0}(\dd)$.
\end{definition}

\cref{fig:Trajectory,fig:Trajectory2} illustrate the preceding definitions. The random billiard models in \cref{def:RRBT}, as well as the other random billiard systems that we will introduce later, appear to be genuinely new, and they suggest numerous directions for future work (see \cref{sec:other}). However, we note that there are several previous works that have considered other (quite different) stochastic billiard systems (see, e.g., \cite{Cook, Comets, Feres}). In addition, a recent article by Defant, Jiradilok, and Mossel \cite{DJM} considers a variant of the random combinatorial billiard process considered here in which the ``reduced'' aspect of the model is removed. In that model, when the beam of light hits a hyperplane in $\mathcal H_{\widetilde W}$, it passes through with probability $p$ and reflects with probability $1-p$, regardless of whether it has already crossed the hyperplane.

\begin{figure}[ht]
\begin{center}{\includegraphics[width=0.5\linewidth]{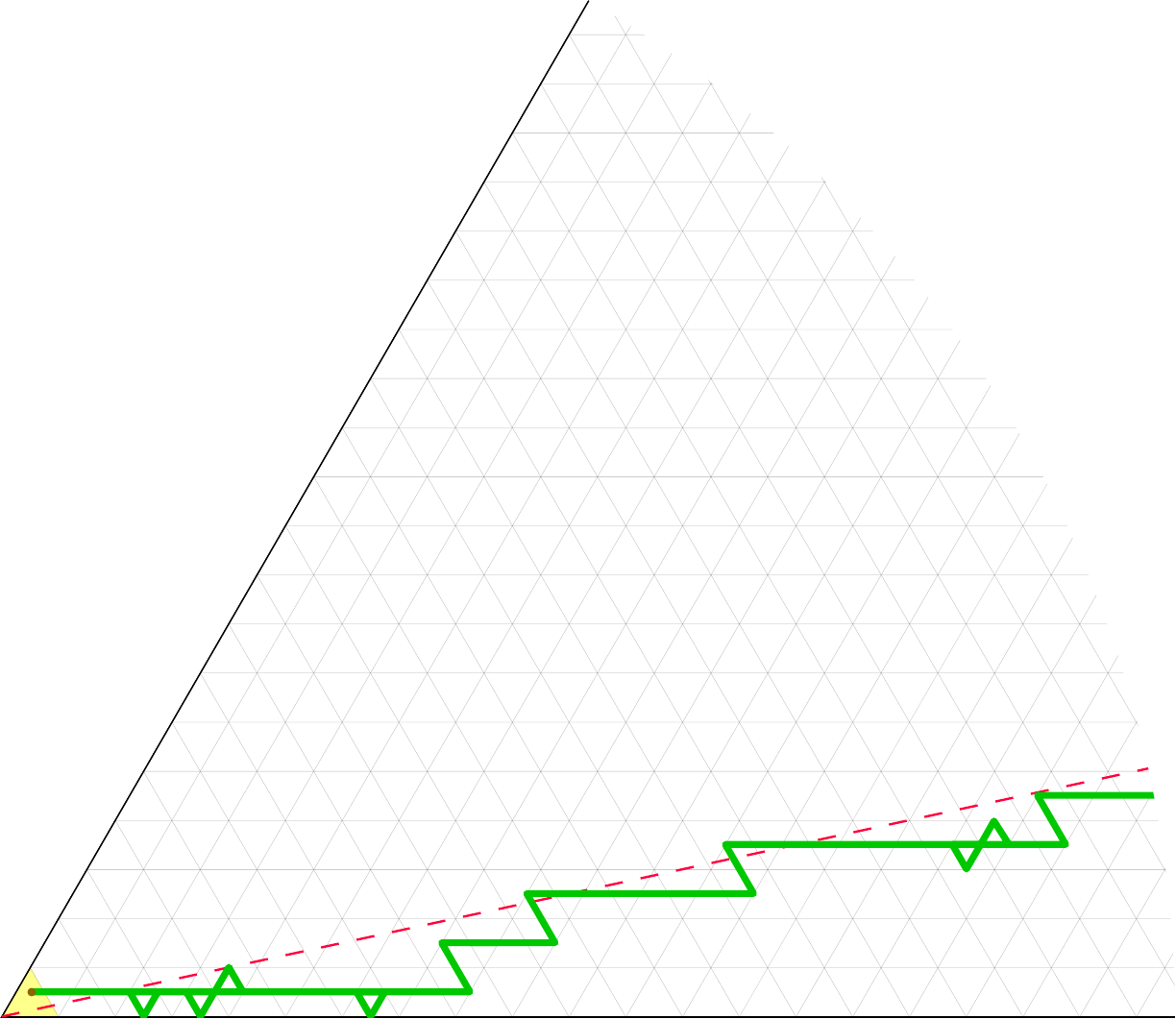}}
\end{center}
\caption{An affine Grassmannian reduced random billiard trajectory in ${\widehat\SSS_3=\widehat A_2}$ with parameter $p=3/4$. The ({\color{MildGreen}green}) beam of light starts in the ({\color{Yellow}yellow}) fundamental alcove traveling in the direction of the vector $\ddd^{(3)}=(1,1,-2)$. The {\color{red}red} dotted ray indicates the direction that the beam of light follows asymptotically.}   
\label{fig:Trajectory2}
\end{figure} 

We can discretize the reduced random billiard trajectory by only keeping track of the alcove containing the beam of light and the direction that the beam of light is facing. Let $u_M$ be the alcove containing the beam of light after it hits a hyperplane in $\mathcal H_{\widetilde W}$ for the $M$-th time; at this point in time, the beam of light is facing toward the facet of $\BB u_M$ contained in the hyperplane $\HH^{(u_M,s_{i_M})}$.  In this way, we obtain a discrete-time Markov chain $\widetilde{\bf M}_\dd$ whose state at time $M$ is the pair $(u_M,M)$ in $\widetilde W\times\ZZ/N\ZZ$. We call this Markov chain a \dfn{reduced random combinatorial billiard trajectory} (RRCBT). In a similar manner, we can discretize the affine Grassmannian reduced random billiard trajectory to obtain the \dfn{affine Grassmannian reduced random combinatorial billiard trajectory} (AGRRCBT), which is a discrete-time Markov chain $\widehat{\bf M}_\dd$ with state space $\widehat W\times\ZZ/N\ZZ$. By projecting $\widehat{\bf M}_{\dd}$ through the natural quotient map ${\widetilde W\times \ZZ/N\ZZ\to W\times \ZZ/N\ZZ}$ given by $(w,M)\mapsto (\overline w,M)$, we obtain a Markov chain ${\bf M}_{\dd}$ on $W\times\ZZ/N\ZZ$. We will always impose the (mild) assumption that ${\bf M}_\dd$ is irreducible (this will be the case in all of the specific examples of interest to us). The Markov chain ${\bf M}_\dd$ can be seen as a random combinatorial billiard trajectory in the torus $\mathbb T$. Each toric hyperplane of the form $\qqq(\HH)$ for $\HH\in\mathcal H_{\widetilde W}$ has two sides. When the beam of light (now traveling in the torus) hits $\qqq(\HH)$, it either passes through or reflects; the probability that it passes through is either $p$ or $0$, depending on which side of $\qqq(\HH)$ it hits (see \cref{fig:torus_directions}). Let $\zeta_{\dd}$ denote the stationary probability distribution of ${\bf M}_{\dd}$. We stress that $\zeta_\dd$ depends on the fixed parameter $p$.

\begin{figure}[ht]
\begin{center}{\includegraphics[height=8.26cm]{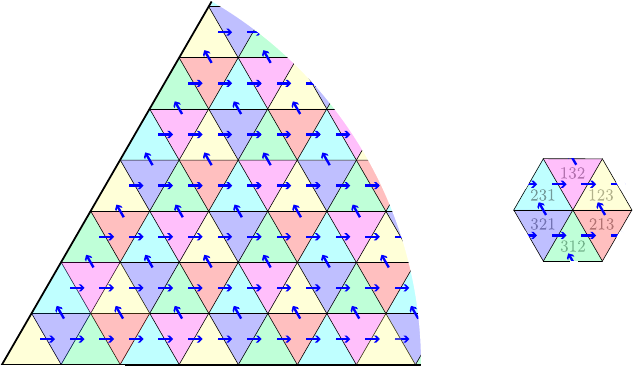}}
\end{center}
\caption{In an AGRRCBT in $\widehat\SSS_3$ defined with respect to the initial direction $\ddd^{(3)}=(1,1,-2)\in Q^\vee$ (which points to the right in the image on the left), the light can only pass through hyperplanes in the directions indicated by arrows in the left image. More precisely, when the light tries to move along one of the arrows, it does so with probability $p$ (and reflects with probability $1-p$). Two alcoves on the left are given the same color if and only if they are related by translation by a coroot vector. A fundamental domain for the torus $\mathbb T$ is shown on the right; note that opposite sides of this hexagon are identified. The toric alcoves correspond to the permutations in $\SSS_3$. The AGRRCBT projects to this torus. In the torus, the light can only pass through toric hyperplanes in the directions indicated by arrows. }   
\label{fig:torus_directions}
\end{figure}

The following two theorems provide analogues of Lam's \cref{thm:Lam}. 

\begin{theorem}\label{thm:LamBilliards} 
Let $\dd\in\Upsilon_{z_0}$, and let $\mathsf{w}(\dd)=\cdots s_{i_1}s_{i_0}$. Let \[(u_M,M)\in \widetilde W\times \ZZ/N_\dd\ZZ\quad\text{and}\quad(v_M,M)\in\widehat W\times\ZZ/N_\dd\ZZ\] denote the states of $\widetilde{\bf M}_\dd$ and $\widehat{\bf M}_{\dd}$, respectively, at time $M$. Let 
\[\psi_\dd=\sum_{\substack{(w,k)\in W\times \ZZ/N_\dd\ZZ\\ i_k=0 \\ w^{-1}\theta\in\Phi^+}}\zeta_\dd(w,k)\theta^\vee w.\] With probability $1$, 
\begin{equation}\label{eq:AGpsi}
\lim\limits_{M\to\infty}\langle \BB v_M\rangle=\langle\psi_\dd\rangle.
\end{equation} 
Moreover, with probability $1$, the limit $\lim\limits_{M\to\infty}\langle \BB u_M\rangle$ exists and belongs to $\langle\psi_\dd\rangle W$. 
\end{theorem} 

\begin{theorem}\label{thm:LamBilliards2} 
Let $\dd\in\Upsilon_{z_0}$. Choose $\epsilon\in\{\pm 1\}$ uniformly at random, and let $(u_M,M)\in\widetilde W\times\ZZ/N_\dd\ZZ$ denote the state of $\widetilde{\bf M}_{\epsilon\dd}$ at time $M$. There exists $w_*\in W$ such that $u_M\in\CC w_*$ for all sufficiently large $M$. For each $w\in W$, we have 
\[\mathbb P(w_*=w)=\frac{N_\dd}{2}(\zeta_{\dd}(w^{-1}\wo,0)+\zeta_{-\dd}(w^{-1}\wo,0)).\]
\end{theorem} 

\begin{remark}
There is another natural interpretation of the reduced random combinatorial billiard trajectory in terms of the Demazure product (see \cref{subsec:general} for the definition). Let $(u_M,M)$ be the state of $\widetilde{\bf M}_{\dd}$ at time $M$. Let $\mathsf{w}(\dd)=\cdots s_{i_1}s_{i_0}$, and let $\mathsf{x}_M$ be the word obtained from $s_{i_{M-1}}\cdots s_{i_1}s_{i_0}$ by deleting each letter with probability $1-p$ (all independently). Then $u_M$ has the same distribution as the Demazure product of $\mathsf{x}_M$. This perspective elucidates a similarity between the present article and recent work of Morales, Panova, Petrov, and Yeliussizov \cite{MPPY}. Their main object of study is a random permutation obtained by starting with a particular reduced word in the symmetric group $\mathfrak{S}_n$, deleting letters randomly and independently, and then taking the Demazure product of the resulting subword (see also \cite{DefantPermuton}). 
\end{remark}   

\begin{remark}
One can view the aformentioned Markov chains introduced by Lam in \cite{Lam} as limits of our ``billiardized'' Markov chains in the regime when $p$ tends to $0$. 
\end{remark}

\subsection{Type~$A$}\label{subsec:introA} 

Assume now that $W$ and $\widetilde W$ are the symmetric group $\SSS_n$ and the affine symmetric group $\affS_n$. In this case, Lam \cite{Lam} observed/conjectured that his reduced random walk has especially nice properties. 

We can identify the index set $\widetilde I$ with $\ZZ/n\ZZ$ in such a way that $s_is_{i+1}s_i=s_{i+1}s_is_{i+1}$ for all $i\in\ZZ/n\ZZ$. Let $e_i$ be the $i$-th standard basis vector in $\mathbb R^n$. Then $\Phi=\Phi^+\sqcup\Phi^-$, where 
\[\Phi^+=\{e_i-e_j:1\leq i<j\leq n\}\quad\text{and}\quad\Phi^-=\{e_j-e_i:1\leq i<j\leq n\}.\] The spaces $V$ and $V^*$ can each be identified with 
\[\{(\gamma_1,\ldots,\gamma_n)\in\mathbb R^n:\gamma_1+\cdots+\gamma_n=0\}.\] We have \[\HH_{e_i-e_j}^k=\{(\gamma_1,\ldots,\gamma_n)\in V^*:\gamma_i-\gamma_j=k\}.\] 

Lam's Markov chain ${\bf M}_{\Lam}$ is isomorphic (in type~$A$) to an instance of a well-studied interacting particle system known as the \emph{multispecies TASEP}, which probabilists and statistical physicists began studying long before Lam's work \cite{AAMP, Derrida, EFM, FFK, FerrariMartin, Prolhac, Spitzer}. In \cite{Lam}, Lam conjectured that the vector $\psi_\Lam$ from \cref{thm:Lam} is a positive scalar multiple of $\sum_{\beta\in\Phi^+}\beta^\vee$ (this can fail to hold outside of type~$A$). By analyzing correlations in the multispecies TASEP, Ayyer and Linusson \cite{AyyerLinusson} proved the following result, which is a reformulation of Lam's conjecture. 

\begin{theorem}[{Ayyer--Linusson~\cite{AyyerLinusson}}]\label{thm:AyyerLinusson}
The vector $\psi_\Lam$ is a positive scalar multiple of 
\[\sum_{1\leq i<j\leq n}(j-i)(e_i-e_j).\] 
\end{theorem}

Let $[n]=\{1,\ldots,n\}$. Let $\ddd^{(n)}=-ne_n+\sum_{j\in[n]}e_j$ be the vector in $V^*$ whose last component is $-(n-1)$ and whose other components are all equal to $1$. As before, fix a point $z_0$ in the interior of $\BB$. One can show that $\ddd^{(n)}\in\Upsilon_{z_0}$. Moreover, $N_{\ddd^{(n)}}=n$, and
$\mathsf{w}(\ddd^{(n)})=\cdots s_{i_1}s_{i_0}=\cdots\mathsf{c}\mathsf{c}\mathsf{c}$, where $\mathsf{c}=s_{n-1}\cdots s_{1} s_{0}$ is a reduced word for a particular \emph{Coxeter element} of $\affS_n$. Thus, $i_j=j\in\ZZ/n\ZZ$ for all $j\geq 0$. 

Ferrari and Martin \cite{FerrariMartin} described the stationary distribution of the multispecies TASEP in terms of \emph{multiline queues}, and Corteel, Mandelshtam, and Williams \cite{CMW} (building off of work of Cantini, de Gier, and Wheeler \cite{CdGW}) interpreted this distribution in terms of specializations of certain \emph{ASEP polynomials}. In particular, their results can be used to compute the distribution $\zeta_\Lam$. (Lam was apparently unaware of the work by Ferrari and Martin when he wrote his article \cite{Lam}.) \cref{thm:LamBilliards,thm:LamBilliards2} motivate us to study ${\bf M}_{{\ddd^{(n)}}}$, which we view as a ``billiardization'' of ${\bf M}_\Lam$. In \cref{sec:ring}, we will 
define a new variant of the multispecies TASEP that we call the \emph{stoned multispecies TASEP}. Surprisingly, we can compute the stationary distribution of this Markov chain in terms of multiline queues and ASEP polynomials. In a special case, the stoned multispecies TASEP is (essentially) the same as ${\bf M}_{\ddd^{(n)}}$. We will (quite unexpectedly) be able to use multiline queues to analyze correlations in the stoned multispecies TASEP (\cref{prop:correlations}). This will in turn allow us to obtain the following analogue of Ayyer and Linusson's \cref{thm:AyyerLinusson}. 

\begin{theorem}\label{thm:AyyerLinussonBilliards} 
The vector $\psi_{\ddd^{(n)}}$ is a positive scalar multiple of 
\[\sum_{1\leq i<j\leq n}\frac{(j-i)(2n-(i+j-1)p)}{(n-ip)(n-(i-1)p)(n-jp)(n-(j-1)p)}(e_i-e_j).\]
\end{theorem}

Note that sending $p$ to $0$ in \cref{thm:AyyerLinussonBilliards} recovers Ayyer and Linusson's \cref{thm:AyyerLinusson}. 

By expressing $\zeta_{{\ddd^{(n)}}}$ and $\zeta_{-{\ddd^{(n)}}}$ in terms of ASEP polynomials, we will also make \cref{thm:LamBilliards2} more explicit when $\dd=\ddd^{(n)}$. We refer to \cref{subsec:setup_ring} for the definitions of the ASEP polynomials $F_\mu$ and Macdonald polynomials $P_\lambda$ appearing in the next theorem; see also \cref{subsec:multiline_queues} for combinatorial interpretations of these polynomials in terms of multiline queues. Note that, throughout this article, the parameter $q$ traditionally appearing in ASEP polynomials and Macdonald polynomials is set to equal $1$ and is therefore omitted from the notation. 

\begin{theorem}\label{cor:LamBilliards2}
Let $F_\mu$ and $P_\lambda$ denote the ASEP polynomial indexed by $\mu$ and the Macdonald polynomial indexed by $\lambda$, respectively. Choose $\epsilon\in\{\pm 1\}$ uniformly at random, and let $(u_M,M)\in\affS_n\times\ZZ/n\ZZ$ denote the state of $\widetilde{\bf M}_{\epsilon\ddd^{(n)}}$ at time $M$. There exists $w_*\in W$ such that $u_M\in\CC w_*$ for all sufficiently large $M$. For each $w\in W$, we have 
\[\mathbb P(w_*=w)=\frac{1}{2}\left(\frac{F_{w^{-1}\wo}(1,\ldots,1,1-p;0)}{P_{(1,2,\ldots,n)}(1,\ldots,1,1-p;0)}+\frac{F_{w^{-1}\wo}(1,\ldots,1,(1-p)^{-1};0)}{P_{(1,2,\ldots,n)}(1,\ldots,1,(1-p)^{-1};0)}\right).\]
\end{theorem} 

\begin{example}
Let $n=3$ so that $\ddd^{(3)}=(1,1,-2)$. Using \cref{thm:AyyerLinussonBilliards}, we compute that \[\langle\psi_{\ddd^{(n)}}\rangle=\langle(3-2p,p,p-3)\rangle.\] \cref{fig:Trajectory} illustrates this when $p=3/4$; the six red dotted rays point in the directions of the vectors in $\langle(1.5,0.75,-2.25)\rangle\SSS_3$. Note that (by \cref{thm:AyyerLinusson}) \[\lim_{p\to 0}\langle\psi_{\ddd^{(n)}}\rangle=\langle(3,0,-3)\rangle=\langle\psi_\Lam\rangle.\] 
As in \cref{thm:LamBilliards2}, let us choose $\epsilon\in\{\pm 1\}$ uniformly at random and define $w_*$ to be the element of $\SSS_3$ such that the reduced random billiard trajectory with initial direction $\epsilon{\ddd^{(n)}}$ (and starting point $z_0$) eventually stays within $\CC w_*$. One can use \cref{cor:LamBilliards2} to compute that 
\begin{alignat*}{3}
\mathbb P(w_*=123)&=\frac{2-2p+p^2}{18-18p+4p^2},\quad \mathbb P(w_*=132)&&=\frac{2-2p}{9-9p+2p^2},\quad \mathbb P(w_*=213)&&=\frac{4-4p+p^2}{18-18p+4p^2}, \\ \mathbb P(w_*=231)&=\frac{2-2p+p^2}{18-18p+4p^2},\quad \mathbb P(w_*=312)&&=\frac{1-p}{9-9p+2p^2},\quad \mathbb P(w_*=321)&&=\frac{4-4p+p^2}{18-18p+4p^2}.
\end{alignat*}
\end{example} 

\begin{example}
Let $n=4$ so that $\ddd^{(4)}=(1,1,1,-3)$. Using \cref{thm:AyyerLinussonBilliards}, we compute that \[\langle\psi_{\ddd^{(4)}}\rangle=\langle(24-30p+9p^2,8-2p-3p^2,-8+14p-3p^2,-24+18p-3p^2)\rangle.\] Note that (by \cref{thm:AyyerLinusson}) \[\lim_{p\to 0}\langle\psi_{\ddd^{(4)}}\rangle=\langle(24,8,-8,-24)\rangle=\langle\psi_\Lam\rangle.\] 
\end{example} 

\subsection{Core Partitions} 
An \dfn{$n$-core} is an integer partition that does not have any hook lengths divisible by $n$. Such partitions are important due to their prominence in partition theory \cite{Armstrong,BrunatNath,Garvan}, the (co)homology of the affine Grassmannian \cite{LLMS}, tiling theory \cite{DLPY, Pak}, and representation theory \cite{GranvilleOno, JamesKerber}.  

There is a natural one-to-one correspondence between $n$-cores and alcoves of $\mathcal H_{\affS_n}$ inside the fundamental chamber $\CC$. Using this correspondence, Lam interpreted his affine Grassmannian reduced random walk as a random growth process for $n$-cores and also as a periodic variant of the TASEP on $\ZZ$. He showed that his conjecture about the specific form of $\psi_\Lam$, which later became Ayyer and Linusson's \cref{thm:AyyerLinusson}, implies an exact description of limit shapes of the (appropriately scaled) Young diagrams in this random growth process. As $n$ tends to $\infty$, these limit shapes converge to the region 
\[{\bf R}=\{(x,y)\in\mathbb R^2:y\leq 0\leq x,\, \sqrt{x}+\sqrt{-y}\leq 6^{1/4}\};\] see \cite{AyyerLinusson,Lam}. Note that ${\bf R}$ is also the limit shape that Rost derived for the \emph{corner growth process}, a more classical random growth process for partitions corresponding to the TASEP on $\ZZ$ \cite{Romik,Rost}.  

In \cref{sec:correlations}, we will interpret $\widehat {\bf M}_{\ddd^{(n)}}$ as a random growth process for $n$-cores and as a variant of Johansson's particle process \cite{Johansson}. Using \cref{thm:AyyerLinussonBilliards}, we will obtain an exact description of the limit shape of our random growth process. We will also see that as $n\to\infty$, these limit shapes converge to the region 
\[{\bf R}_\infty^{(p)}=\{(x,y)\in\mathbb R^2:y\leq 0\leq x,\, \sqrt{(1-p)x}+\sqrt{-y}\leq (6(1-p))^{1/4}\}.\] The remarkably simple form of the region ${\bf R}_\infty^{(p)}$ is ultimately due to some especially nice properties of multiline queues. 

In the limit as $n\to\infty$, our random growth process yields the TASEP with gemetric jumps, which has been introduced under a variety of names in the literature and appears in other articles such as \cite{DefantPermuton,MPPY,VK,Rajewsky,Rakos}.  

\subsection{Stoned Exclusion Processes}

While we have so far discussed ways to ``billiardize'' results due to Lam \cite{Lam} and Ayyer--Linusson \cite{AyyerLinusson}, there are further goals of the present article that come from generalizing and modifying the random toric combinatorial billiard trajectory ${\bf M}_{{\ddd^{(n)}}}$. As an initial step, we can introduce a parameter $t\in[0,1)$ and change the dynamics so that when the beam of light hits a toric hyperplane $\qqq(\HH)$, it passes through with probability either $p$ or $pt$, depending on the side of $\qqq(\HH)$ that it hits (setting $t=0$ recovers ${\bf M}_{{\ddd^{(n)}}}$). Alternatively, we could choose positive real parameters $a_1,\ldots,a_n$ and modify the dynamics so that the probability of the light passing through a hyperplane $\qqq(\HH_{i,j}^k)$ (with $1\leq i<j\leq n$) depends on $a_j$. We could also consider a different variant in which we take $\widetilde W$ to be the affine Weyl group of type~$C$. All of these variants (which we will define rigorously in \cref{subsec:billiards1,subsec:billiards2,subsec:billiards3}) can be seen as billiardizations of interacting particle systems that have already received significant attention. We will find that the stationary distributions of these Markov chains are given by important polynomials from the literature. 

By viewing our billiardized particle systems purely combinatorially, we can generalize (and slightly modify) them further to what we call \emph{stoned exclusion processes}. These are variants of vigorously-studied stochastic models in which particles hop along a $1$-dimensional lattice subject to the constraint that two or more particles cannot occupy the same site. We will find that the stationary distributions of these stoned exclusion processes are governed by polynomials from the literature. In several cases, it was previously an open problem to find Markov chains with stationary distributions given by those polynomials. We view this as an endorsement of combinatorial billiards since it was the contemplation of combinatorial billiards that led us naturally to solutions to these problems.  

One of the prototypical examples of an interacting particle system is the \emph{asymmetric simple exclusion process} (ASEP), which dates back to the incredibly influential work of Spitzer \cite{Spitzer}. Building off of work of Martin \cite{Martin} and Cantini, de Gier, and Wheeler \cite{CdGW}, Corteel, Mandelshtam, and Williams \cite{CMW} introduced \emph{ASEP polynomials}, which are certain polynomials in variables $x_1,\ldots,x_n$ whose coefficients are rational functions in a parameter $t$.\footnote{ASEP polynomials are also defined with an additional parameter $q$. However, we will only be concerned with the case in which $q=1$. Thus, whenever we consider ASEP polynomials or their variants, we will do so without mentioning $q$.} When ${x_1=\cdots=x_n=1}$, these articles show that ASEP polynomials yield the stationary distribution of the multispecies ASEP with $n$ sites and that the partition function of this multispecies ASEP is a Macdonald polynomial. Corteel, Mandelshtam, and Williams also provided combinatorial expansions of ASEP polynomials using multiline queues, and they found that ASEP polynomials are actually special instances of the \emph{permuted-basement Macdonald polynomials} introduced by Ferreira~\cite{Ferreira}. 
  
It was an open problem to find a variant of the multispecies ASEP whose stationary distribution is determined by ASEP polynomials with generic choices of parameters $x_1,\ldots,x_n$ rather than the specialization $x_1=\cdots=x_n=1$. By generalizing the Markov chain ${\bf M}_{{\ddd^{(n)}}}$ from \cref{subsec:introA}, we will define the \emph{stoned multispecies ASEP}, which has the desired stationary distribution (\cref{thm:main_ASEP}). Loosely speaking, a state of this Markov chain is obtained by superimposing a state from the multispecies ASEP with a state from a multispecies TASEP that we call the \emph{auxiliary TASEP}; we refer to the objects that hop in the auxiliary TASEP as \emph{stones} in order to avoid confusion with the \emph{particles} in the multispecies ASEP (the use of stones comes from the articles \cite{DefantRefractions,DefantPermutoric}). The dynamics of the stoned multispecies ASEP is as follows. The auxiliary TASEP runs as a usual multispecies TASEP. Whenever two stones sitting on adjacent sites swap places, they send a signal to the particles sitting on those same sites saying that they should swap. The signal succeeds in actually reaching the particles with some probability that depends on the stones that swapped; the particles do not move if they do not receive the signal. If the signal does reach the particles, then with some probability (depending on which of the particles has a larger species), they decide to actually follow their orders and swap places (with the complementary probability, they stubbornly ignore the signal and do not move). 
In summary, while the particle swaps in the usual multispecies ASEP are governed by the selection of uniformly random edges of the ring, the particle swaps in the stoned multispecies ASEP are governed by the stone swaps in the auxiliary TASEP. 

We will also define stoned versions of other exclusion processes such as the \emph{inhomogeneous TASEP} and the \emph{multispecies open boundary ASEP}. We will show that the stoned inhomogeneous TASEP has a stationary distribution given by certain polynomials introduced by Cantini \cite{Cantini} that we call \emph{inhomogeneous TASEP polynomials} (see \cref{thm:main_iTASEP}). We will also show that the stoned multispecies open boundary ASEP has a stationary distribution given by the \emph{open boundary ASEP polynomials} from \cite{CGdGW, CMW2} (see \cref{thm:main_obASEP}). It was previously not known how to construct Markov chains with these stationary distributions. 

Our derivation of the stationary distribution of each of our stoned exclusion processes is actually quite simple. This is because, roughly speaking, the stones serve the purpose of ``locally factoring'' the transitions from the ordinary version of the exclusion process. In this way, the Markov chain balance equations end up being essentially equivalent to known relations among the polynomials that determine the stationary distribution. 

\begin{remark}\label{rem:tPush}
Ayyer, Martin, and Williams \cite{AMW} recently studied a Markov chain called the \emph{inhomogeneous $t$-PushTASEP}, which is quite different from our stoned exclusion processes (see also the earlier work of Ayyer and Martin \cite{AyyerMartinNew} in the case where $t=0$); they found that its stationary distribution is also given by ASEP polynomials with generic choices of $x_1,\ldots,x_n$. We discovered stoned exclusion processes independently of their work while considering random billiard trajectories. A major advantage of our approach using stones is that it easily adapts to the inhomogeneous TASEP and the multispecies open boundary ASEP; it is not known how to adapt the $t$-PushTASEP to those settings. 
\end{remark} 

\subsection{Outline} Our plan for the rest of the paper is as follows. 

\begin{itemize}
\item \cref{sec:preliminaries} presents additional background that we will need in subsequent sections. 

\item \cref{sec:LamProof} presents the proofs of \cref{thm:LamBilliards,thm:LamBilliards2}. 

\item In \cref{sec:ring}, we discuss the multispecies ASEP and define its stoned variant. Our main results of this section (\cref{thm:main_ASEP,cor:main_ASEP}) compute the stationary distribution of the stoned multispecies ASEP in terms of ASEP polynomials. We also explain how to interpret a special case of the stoned multispecies ASEP in terms of billiards, and we deduce \cref{cor:LamBilliards2}. 

\item In \cref{sec:correlations}, we define multiline queues and employ them to study correlations in the stoned multispecies TASEP; this allows us to prove \cref{thm:AyyerLinussonBilliards} and to compute limit shapes of our new random growth process on $n$-cores. 

\item In \cref{sec:inhomogeneous}, we consider the inhomogeneous TASEP. We define the stoned inhomogeneous TASEP, compute its stationary distribution in terms of inhomogeneous TASEP polynomials (\cref{thm:main_iTASEP}), and discuss how to interpret a special case of it in terms of billiards. 

\item In \cref{sec:open}, we consider the multispecies open boundary ASEP. We define the stoned multispecies open boundary ASEP, compute its stationary distribution in terms of open boundary ASEP polynomials (\cref{thm:main_obASEP}), and discuss how to interpret a special case of it in terms of billiards in type~$C$. 

\item As a general philosophy, we believe that combining randomness with notions from combinatorial billiards leads to fruitful yet unexplored stochastic models, several of which are natural variants/generalizations of well-studied models. \cref{sec:other} exemplifies this philosophy by suggesting numerous promising directions for future work. 
\end{itemize}

\section{Preliminaries}\label{sec:preliminaries}

\subsection{Notation} 
Throughout this work, we fix $t\in[0,1)$ and a tuple ${\lambda=(\lambda_1,\ldots,\lambda_n)\in\mathbb Z^n}$ such that ${0\leq\lambda_1\leq\cdots\leq\lambda_n}$. Let $S_{\lambda}$ be the set of tuples that can be obtained by rearranging the parts of $\lambda$. We always let $\mu_i$ denote the $i$-th part of a tuple $\mu$. We let ${\bf x}$ denote the tuple of variables $(x_1,\ldots,x_n)$. Given integers $k,k'\in\ZZ$ and $r\in\mathbb R$, let 
\begin{equation}\label{eq:h}
\hh_r(k,k')=\begin{cases} 1 & \mbox{if }k> k'; \\   r & \mbox{if }k<k'; \\   0 & \mbox{if }k=k'.\end{cases}
\end{equation} 

\subsection{General Weyl Groups}\label{subsec:general} 

Let us expound on the discussion from \cref{subsec:IntroWeyl}. We assume familiarity with the basic theory of root systems and Weyl groups; standard references include \cite{BjornerBrenti,Humphreys}.  

By definition, the Weyl group $W$ is the group generated by the reflections through the hyperplanes orthogonal to the roots in $\Phi$. Thus, we have $w\beta\in\Phi$ for all $w\in W$ and $\beta\in\Phi$. If we identify the spaces $V$ and $V^*$ via the usual pairing, then the coroots are defined by 
\[\beta^\vee=\frac{2}{(\beta,\beta)}\beta,\]
where $(\cdot,\cdot)$ is the Euclidean inner product. 
Let $Q^\vee=\mathrm{span}_{\ZZ}\{\beta^\vee:\beta\in\Phi\}$ denote the coroot lattice. The space $\mathbb T=V^*/Q^\vee$ is a torus of dimension $|I|$; we let $\qqq\colon V^*\to\mathbb T$ denote the natural quotient map. Applying this quotient map to a hyperplane $\HH\in\mathcal H_{\widetilde W}$ yields a \dfn{toric hyperplane} $\qqq(\HH)$. For $\dd\in Q^\vee$, let $\tau_\dd$ denote the element of $\widetilde W$ that acts on $V^*$ via the translation by $-\dd$; that is, $\gamma\tau_\dd=\gamma-\dd$ for all $\gamma\in V^*$. The injective map $\dd\mapsto\tau_\dd$ allows us to identify $Q^\vee$ with a subgroup of $\widetilde W$. In this way, $\widetilde W$ decomposes as the semidirect product $\widetilde W\cong W\ltimes Q^\vee$. Thus, for each $w\in \widetilde W$, there are unique elements $\overline w\in W$ and $\dd_w\in Q^\vee$ such that $w=\overline w\tau_{\dd_w}$. In particular, $\dd_{s_0}=-\theta^\vee$, so 
\begin{equation}\label{eq:s0theta}
s_0=\overline s_0\tau_{-\theta^\vee}. 
\end{equation}
Note that the alcove $\BB w$ is obtained by translating $\BB\overline w$ by the vector $-\dd_w$.  

Let $\idd$ denote the identity element of $\widetilde W$ (and also of $W$). For $i,i'\in\widetilde I$, let $m(i,i')=m(i',i)$ denote the order of the element $s_{i}s_{i'}$ in $\widetilde W$. Since $W$ and $\widetilde W$ are Coxeter groups, they have the \dfn{Coxeter presentations} 
\[W=\langle S:(s_is_{i'})^{m(i,i')}=\idd\text{ for all }i,i'\in I\rangle\quad\text{and}\quad \widetilde W=\langle \widetilde S:(s_is_{i'})^{m(i,i')}=\idd\text{ for all }i,i'\in \widetilde I\rangle.\] 
A \dfn{reduced word} for an element $w\in\widetilde W$ is a minimum-length word over the alphabet $\widetilde S$ that represents $w$. The length of a reduced word for $w$ is called the \dfn{length} of $w$. The finite Weyl group $W$ has a unique element $\wo$ of maximum length. It is well known that $\wo$ is an involution. 

Given a positive integer $d$ and symbols $\phi$ and $\phi'$, we write 
\begin{equation}\label{eq:mid}
[\phi\mid\phi']_d=\underbrace{\cdots\phi'\phi\phi'}_{d}
\end{equation} for the string of length $d$ that alternates between $\phi$ and $\phi'$ and ends with $\phi'$. 
The \dfn{$0$-Hecke algebra} of $\widetilde W$ is the $\mathbb C$-algebra with generators $\TT_i$ for $i\in\widetilde I$ subject to the following relations: 
\begin{alignat*}{2}
\TT_i^2&=\TT_i \quad &&\text{for all }i\in\widetilde I; \\ 
[\TT_i\mid \TT_{i'}]_{m(i,i')}&=[\TT_{i'}\mid \TT_i]_{m(i,i')} \quad &&\text{for all }i,i'\in\widetilde I. 
\end{alignat*} 
For $w\in\widetilde W$, let $\TT_w=\TT_{j_\ell}\cdots\TT_{j_1}$, where $s_{j_\ell}\cdots s_{j_1}$ is a reduced word for $w$; the defining relations of the $0$-Hecke algebra ensure that $\TT_w$ is well defined (i.e., does not depend on the choice of the reduced word for $w$). The elements $\TT_w$ for $w\in \widetilde W$ form a linear basis of the $0$-Hecke algebra of $\widetilde W$. Given an arbitrary finite word $\mathsf{x}=s_{k_r}\cdots s_{k_1}$ over the alphabet $\widetilde S$, there is a unique element $y\in \widetilde W$ such that $\TT_{k_r}\cdots\TT_{k_1}=\TT_y$; this element $y$ is called the \dfn{Demazure product} of the word $\mathsf{x}$.

\subsection{Type~$A$}\label{subsec:typeA}  
Let $e_i$ denote the $i$-th standard basis vector of $\mathbb R^n$. 
The root system of type $A_{n-1}$ is the collection of vectors $e_i-e_j$ for distinct $i,j\in[n]$. In this setting, we have \[V=V^*=\{(\gamma_1,\ldots,\gamma_n)\in\mathbb R^n:\gamma_1+\cdots+\gamma_n=0\}.\] The index set $I$ is $[n-1]$, and the simple root $\alpha_i$ is $e_i-e_{i+1}$. The highest root is $\theta=e_1-e_n$. The coroot lattice is $Q^\vee=\mathbb Z^n\cap V^*$. We identify the Weyl group $W$ with the symmetric group $\SSS_n$; under this identification, the simple reflection $s_i=\overline s_i$ (for $1\leq i\leq n-1$) is the transposition in $\SSS_n$ that swaps $i$ and $i+1$, and $\overline s_0$ is the transposition that swaps $n$ and $1$. The affine Weyl group $\widetilde W$ is the affine symmetric group $\affS_n$ (see \cite[Chapter~8]{BjornerBrenti}). We can identify the index set $\widetilde I$ with $\ZZ/n\ZZ$ so that $s_is_{i+1}s_i=s_{i+1}s_is_{i+1}$ for all $i\in\ZZ/n\ZZ$. We often represent a permutation $w\in\SSS_n$ via its \dfn{one-line notation}, which is the word $w(1)\cdots w(n)$. 

\section{Reduced Random Combinatorial Billiard Trajectories}\label{sec:LamProof}  
Let $\widetilde W$ be an affine Weyl group, and let $W$ be the associated finite Weyl group. Fix a point $z_0$ in the interior of the fundamental alcove $\BB$. Let $\dd\in\Upsilon_{z_0}$, and let $N=N_\dd$ be the period of the word $\mathsf{w}(\dd)=\cdots s_{i_1}s_{i_0}$. We now prove \cref{thm:LamBilliards,thm:LamBilliards2}, which describe the asymptotic behavior of $\widetilde{\bf M}_{\dd}$ in terms of the stationary distributions $\zeta_{\dd}$ and $\zeta_{-\dd}$.

\subsection{Proof of \cref{thm:LamBilliards}} 
 Our argument in this section is quite similar to Lam's proof of the first statement in \cref{thm:Lam}, so we omit some details. We start with the following lemma, which is a simple consequence of the definition of the right action of $W$ on $V^*$. 

\begin{lemma}\label{lem:AG_right_action} 
Let $(u_M,M)$ and $(v_M,M)$ denote the states of $\widetilde{\bf M}_\dd$ and $\widehat{\bf M}_\dd$, respectively, at time $M$. Then $v_M$ has the same distribution as the unique affine Grassmannian element of $u_MW$. 
\end{lemma}

The proof of the second statement in \cref{thm:LamBilliards} follows from \eqref{eq:AGpsi} and \cref{lem:AG_right_action}. Thus, we just need to prove \eqref{eq:AGpsi}. 

Denote by ${(v_M,M)\in\widehat W\times\ZZ/N\ZZ}$ the state of the affine Grassmannian reduced random combinatorial billiard trajectory $\widehat{\bf M}_\dd$ at time $M$. We can write $v_M$ uniquely as $\overline v_M\tau_{\dd_{v_M}}$ with $\overline v_M\in W$ and $\dd_{v_M}\in Q^\vee$. The probability that the sequence $(v_M)_{M\geq 0}$ is eventually constant is $0$, so we will tacitly assume in what follows that it is not eventually constant. This implies that the distance between the vectors $\langle\BB v_M\rangle$ and $\langle-\dd_{v_M}\rangle$ tends to $0$ as $M\to\infty$. Thus, we wish to compute $\lim_{M\to\infty}\langle-\dd_{v_M}\rangle$. 

Every transition of positive probability in ${\bf M}_\dd$ is of the form $(w,k)\to(w,k+1)$ or of the form $(w,k)\to(\overline{s}_{i_k}w,k+1)$ for $w\in W$ and $k\in\ZZ/N\ZZ$. Given an edge $e$ of the form $(w,k)\to(\overline{s}_{i_k}w,k+1)$ and a positive integer $D$, define $\kappa_D(e)$ to be the number of nonnegative integers $M\leq D-1$ such that $M\equiv k\pmod{N}$, $\overline v_M=w$, and $\overline v_{M+1}=\overline{s}_{i_k}w$. Since the transition probability of the edge $e$ is $p$, it follows from the ergodic theorem for Markov chains \cite[Corollary~4.1]{Bremaud} that \begin{equation}\label{eq:lime}
\lim_{D\to\infty}\frac{1}{D}\kappa_D(e)=\zeta_\dd(w,k)p.
\end{equation} 

For each nonnegative integer $M$, we have $\dd_{v_M}\neq\dd_{v_{M+1}}$ if and only if $v_{M+1}=s_0 v_M$; when this is the case, we can use the identity $s_0=\overline{s}_0\tau_{-\theta^\vee}$ (see \eqref{eq:s0theta}) to find that $\dd_{v_{M+1}}=\dd_{v_M}-\theta^\vee\overline{v}_M$. Note that if $v_{M+1}=s_0v_M$, then (because the trajectory can never cross a hyperplane twice and must stay in $\CC$) we must have $\overline v_M^{-1}\theta\in\Phi^+$. Together with \eqref{eq:lime}, this implies that 
\[\lim_{M\to\infty}\langle-\dd_{v_M}\rangle=\langle\psi_\dd\rangle,\]
where \[\psi_\dd=\sum_{\substack{(w,k)\in W\times\ZZ/N\ZZ \\ i_k=0 \\ w^{-1}\theta\in\Phi^+}}\zeta_\dd(w,k)\theta^\vee w.\] This proves \eqref{eq:AGpsi}. 

\subsection{Proof of \cref{thm:LamBilliards2}} 

Because $\rr_\dd$ is the unique ray that starts at the point $z_0$ and passes through the points ${z_0+\dd,z_0+2\dd,z_0+3\dd,\ldots}$, it follows that $\rr_{-\dd}$ is the unique ray that starts at $z_0$ and passes through $z_0-\dd,z_0-2\dd,z_0-3\dd,\ldots$. This immediately implies the following lemma. 

\begin{lemma}\label{lem:makes_sense} 
We have $-\dd\in\Upsilon_{z_0}$, and $\mathsf{w}(-\dd)=\cdots\mathsf{yyy}$, where $\mathsf{y}=s_{i_0}s_{i_{1}}\cdots s_{i_{N-1}}$. 
\end{lemma} 

Choose $\epsilon\in\{\pm 1\}$ uniformly at random, and consider the reduced random combinatorial billiard trajectory $\widetilde{\bf M}_{\epsilon\dd}$. Let $(u_M,M)$ be the state of this Markov chain at time $M$. The fact that the billiard trajectory cannot re-enter a chamber that it has already left implies that there exists $w_*\in W$ such that $u_M\in\CC w_*$ for all sufficiently large $M$. 

Let $\{\TT_i:i\in\widetilde I\}$ be the set of generators of the $0$-Hecke algebra of $\widetilde W$, as defined in \cref{subsec:general}. Recall that the elements $\TT_w$ for $w\in\widetilde W$ form a linear basis of the $0$-Hecke algebra. It follows from \cref{lem:makes_sense} that for each positive integer $k$ and each $x\in \widetilde W$, the probability $\mathbb P(u_{kN}=x)$ is equal to the coefficient of $\TT_x$ in 
\begin{equation}\label{eq:T_x} 
\frac{1}{2}\left(\left((1-p+p\TT_{i_{N-1}})\cdots(1-p+p\TT_{i_0})\right)^k+\left((1-p+p\TT_{i_0})\cdots(1-p+p\TT_{i_{N-1}})\right)^k\right). 
\end{equation} 
There is an antiautomorphism of the $0$-Hecke algebra defined by $\TT_w\mapsto\TT_{w^{-1}}$. This antiautomorphism fixes the element in \eqref{eq:T_x}, so for each positive integer $k$ and each $x\in \widetilde W$, we have 
\begin{equation}\label{eq:xxinverse}
\mathbb P(u_{kN}=x)=\mathbb P(u_{kN}=x^{-1}). 
\end{equation} 

For $v,w\in W$, let $\CC_{w}^v=\{u\in \CC w:\overline u=v\}$. An element $x=\overline x\tau_\xi\in\widetilde W$ is \dfn{regular} if there does not exist a nonidentity element $w\in W$ such that $\xi w=\xi$. The following lemma is due to Lam. 

\begin{lemma}[{Lam~\cite[Lemma~7]{Lam}}]\label{lem:regular}  
Let $v,w\in W$. If $x\in \CC_{w}^v$ is regular, then $x^{-1}\in\CC_{\wo w v^{-1}}^{v^{-1}}$.  
\end{lemma} 

For $v,w\in W$, it follows from \cref{thm:LamBilliards,lem:AG_right_action} that 
\begin{equation}\label{eq:eta}
\lim_{k\to\infty}\mathbb P(u_{kN}\in\CC_{w}^v)=\frac{1}{2}(\zeta_{\dd}(vw^{-1},0)+\zeta_{-\dd}(vw^{-1},0))\mathbb P(w_*=w). 
\end{equation} 

We know by \cref{thm:Lam} that (with probability $1$) the limit  $\lim_{k\to\infty}\langle\BB u_{kN}\rangle$ exists and belongs to $\langle\psi_\dd\rangle W\cup\langle\psi_{-\dd}\rangle W$; this implies that $u_{kN}$ is almost surely regular for sufficiently large $k$. 
Therefore, for $v,w\in W$, we can combine \cref{lem:regular} with \eqref{eq:xxinverse} and \eqref{eq:eta} to find that 
\begin{align*} 
\frac{1}{2}(\zeta_{\dd}(vw^{-1},0)+\zeta_{-\dd}(vw^{-1},0))\mathbb P(w_*=w)&=\lim_{k\to\infty}\mathbb P(u_{kN}\in\CC_w^v) \\ 
&= \lim_{k\to\infty}\mathbb P(u_{kN}\in\CC_{\wo wv^{-1}}^{v^{-1}}) \\ 
&= \frac{1}{2}(\zeta_{\dd}(w^{-1}\wo,0)+\zeta_{-\dd}(w^{-1}\wo,0))\mathbb P(w_*=\wo wv^{-1}).
\end{align*} 
By setting $v=\wo w$, we find that 
\[\mathbb P(w_*=w)=\frac{\mathbb P(w_*=\idd)}{\zeta_{\dd}(\wo,0)+\zeta_{-\dd}(\wo,0)}(\zeta_{\dd}(w^{-1}\wo,0)+\zeta_{-\dd}(w^{-1}\wo,0)).\] Since $\sum_{w\in W}\mathbb P(w_*=w)=1$ and $\sum_{w\in W}(\zeta_{\dd}(w^{-1}\wo,0)+\zeta_{-\dd}(w^{-1}\wo,0))=\frac{2}{N}$, we have \[\frac{\mathbb P(w_*=\idd)}{\zeta_{\dd}(\wo,0)+\zeta_{-\dd}(\wo,0)}=\frac{N}{2}.\] We deduce that  
\[\mathbb P(w_*=w)=\frac{N}{2}(\zeta_{\dd}(w^{-1}\wo,0)+\zeta_{-\dd}(w^{-1}\wo,0)),\] which completes the proof of \cref{thm:LamBilliards2}.

\section{The Stoned Multispecies ASEP}\label{sec:ring}

\subsection{Set-up}\label{subsec:setup_ring} 
Throughout this section, we take $\Phi$ to be the root system of type~$A_{n-1}$ so that $W$ is the symmetric group $\SSS_n$ and $\widetilde W$ is the affine symmetric group $\affS_n$. As mentioned in \cref{subsec:typeA}, we identify the index set $\widetilde I$ with $\ZZ/n\ZZ$. For $1\leq i\leq n-1$, the simple reflection $s_i=\overline s_i$ is the transposition that swaps $i$ and $i+1$; the element $\overline s_0$ is the transposition that swaps $n$ and $1$. Given a tuple $\boldsymbol{y}=(y_1,\ldots,y_n)$ and a permutation $w\in\SSS_n$, let $w{\boldsymbol{y}}=(y_{w^{-1}(1)},\ldots,y_{w^{-1}}(n))$. Given a polynomial ${f=f({\bf x})=f(x_1,\ldots,x_n)}$, we define a polynomial $wf$ by letting \[wf({\bf x})=f(w{\bf x})=f(x_{w^{-1}(1)},\ldots,x_{w^{-1}(n)}).\] Let $c\in\SSS_n$ be the permutation whose cycle decomposition is $(1\,\,2\,\cdots\,n)$.

For $i\in\ZZ/n\ZZ$, let $T_i$ denote the operator on $\mathbb C(t)[{\bf x}]$ defined by \[T_i=t-\frac{tx_i-x_{i+1}}{x_i-x_{i+1}}(1-\overline s_i).\] 
One can check that the following relations hold: 
\begin{alignat*}{2}
(T_i-t)(T_i+1)&=0 \quad &&\text{for all }i\in\ZZ/n\ZZ; \\
T_iT_{i+1}T_i&=T_{i+1}T_iT_{i+1} \quad &&\text{for all }i\in\ZZ/n\ZZ; \\ 
T_iT_j&=T_jT_i \quad &&\text{for all }i,j\in\ZZ/n\ZZ\text{ with }j\not\in\{i-1,i,i+1\}. 
\end{alignat*}
(This means that these operators define an action of the \emph{Hecke algebra} of $\affS_n$.) 

There is a unique family $(F_\mu)_{\mu\in S_\lambda}$ of  homogeneous polynomials in $\mathbb C(t)[{\bf x}]$ such that the coefficient of $x_1^{\lambda_n-\lambda_1}x_2^{\lambda_n-\lambda_2}\cdots x_n^{\lambda_n-\lambda_n}$ in $F_\lambda$ is $1$ and such that the following \dfn{exchange equations} hold: 
\begin{alignat}{2}
\label{eq:qKZ1}T_iF_\mu&=F_{\overline{s}_i\mu}\quad&&\text{if }i\in\ZZ/n\ZZ\text{ and }\mu_i<\mu_{i+1}; \\
\label{eq:qKZ2}\overline s_iF_\mu&=F_\mu\quad&&\text{if }i\in\ZZ/n\ZZ\text{ and }\mu_i=\mu_{i+1}; \\ 
\label{eq:qKZ3}cF_{c\mu}&=F_\mu. &&
\end{alignat}
(These equations are obtained by setting $q=1$ in the \emph{$q$KZ equations} from \cite{CMW,KT}.)
These polynomials $F_\mu=F_\mu({\bf x};t)$ are called \dfn{ASEP polynomials}. Cantini, de Gier, and Wheeler \cite{CdGW} showed that the polynomial \[P_\lambda({\bf x})=\sum_{\mu\in S_\lambda}F_\mu({\bf x})\] is a \dfn{Macdonald polynomial}. Corteel, Mandelshtam, and Williams \cite{CMW} reproved this result and gave a combinatorial formula for computing ASEP polynomials using \emph{multiline queues}. While we do not require multiline queues in this section, we will need them in \cref{sec:correlations} (but only in the simplified setting where $t=0$). 

Recall the definition of $\hh_t(k,k')$ from \eqref{eq:h}. 
We let $\AmASEP_\lambda$ denote the \dfn{multispecies ASEP} with state space $S_\lambda$. This is a discrete-time Markov chain in which the transition probability from a state $\mu$ to a state $\mu'$ is given by 
\[\mathbb P(\mu\to\mu')=\begin{cases} \frac{1}{n}\hh_t(\mu_i,\mu_{i+1}) & \mbox{if }\mu'=\overline s_i\mu\neq\mu; \\  1-\sum_{\nu\in S_\lambda\setminus\{\mu\}}\mathbb P(\mu\to\nu) & \mbox{if }\mu=\mu'; \\ 0 & \mbox{otherwise}.  \end{cases}
\]
The state $\mu$ can be visualized as a configuration of particles on a ring with sites $1,\ldots,n$ (listed in clockwise cyclic order), where the particle on site $i$ has \dfn{species} $\mu_i$ (see \cref{fig:mASEP}). There has been substantial attention devoted to the stationary distribution of the multispecies ASEP \cite{AAMP, Derrida, EFM, FFK, FerrariMartin, CdGW, Prolhac, CMW, Martin}. According to \cite{CdGW} (see also \cite{Martin,CMW}), the stationary probability of $\mu$ in $\AmASEP_\lambda$ is \begin{equation}\label{eq:stationaryNew}
\frac{F_\mu(1,\ldots,1;t)}{P_\lambda(1,\ldots,1;t)}.
\end{equation} 
When $t=0$, the multispecies ASEP is called the \dfn{multispecies TASEP}; Ferrari and Martin \cite{FerrariMartin} computed its stationary distribution using multiline queues several years before the more general formula in \eqref{eq:stationaryNew} was discovered.  

\begin{figure}[ht]
\begin{center}{\includegraphics[height=6.761cm]{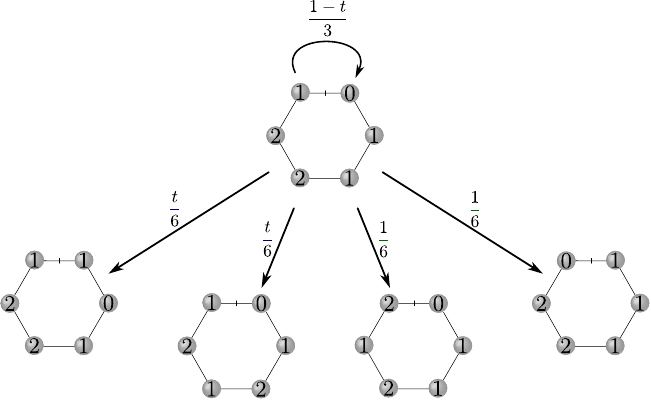}}
\end{center}
\caption{All the transitions from the state $(0,1,1,2,2,1)$ in $\AmASEP_{(0,1,1,1,2,2)}$. In each depiction of the ring, a tick-mark is drawn between sites $6$ and $1$.} 
\label{fig:mASEP}
\end{figure} 

To define the stoned variant of $\AmASEP_\lambda$, we first need to introduce a modified version of the multispecies TASEP. Consider $n$ \dfn{stones} $\s_1,\ldots,\s_n$. Fix an integer $m\geq 2$, and let \[\rrho\colon\{\s_1,\ldots,\s_n\}\to[m]\] be a surjective function such that $1=\rrho(\s_1)\leq\cdots\leq\rrho(\s_n)=m$; we call $\rrho(\s_j)$ the \dfn{density} of $\s_j$. Let $\Omm=\Omm_{\rrho}$ denote the set of permutations $\sigma\in\SSS_n$ such that for every $k\in[m]$, the stones of density $k$ appear in the same cyclic order within the list $\s_{\sigma^{-1}(1)},\ldots,\s_{\sigma^{-1}(n)}$ as they do within the list $\s_1,\ldots,\s_n$.
We can view a permutation $\sigma\in\Omm$ as a certain configuration of the stones on the sites of the ring, where the stone $\s_{j}$ is placed on the site $\sigma(j)$. For $\sigma\in\Omm$, let $\MM(\sigma)$ be the number of indices $i\in\ZZ/n\ZZ$ such that $\rrho(\s_{\sigma^{-1}(i)})<\rrho(\s_{\sigma^{-1}(i+1)})$. The set $\Omm$ is the state space of an irreducible Markov chain called the \dfn{auxiliary TASEP}, whose transition probabilities are given by 
\begin{equation}\label{eq:aux_transitions}
\mathbb P(\sigma\to \sigma')=\begin{cases} \frac{1}{n} & \mbox{if }\sigma'=\overline{s}_i\sigma\text{ and }\rrho(\s_{\sigma^{-1}(i)})<\rrho(\s_{\sigma^{-1}(i+1)}); \\   1-\frac{\MM(\sigma)}{n} & \mbox{if }\sigma=\sigma'; \\ 0 & \mbox{otherwise}.  \end{cases}
\end{equation} 
See \cref{fig:stonesTASEP1}. 

\begin{figure}[ht]
\begin{center}{\includegraphics[height=7.352cm]{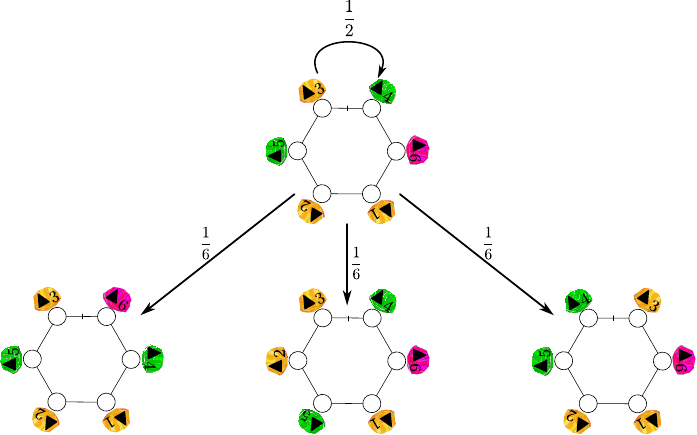}}
\end{center}
\caption{All the transitions from the state $\sigma=346152$ (in one-line notation) of the auxiliary TASEP with $n=6$. The densities of stones are represented by colors; we have $\rrho(\s_1)=\rrho(\s_2)=\rrho(\s_3)=1$, $\rrho(\s_4)=\rrho(\s_5)=2$, and $\rrho(\s_6)=3$.}
\label{fig:stonesTASEP1}
\end{figure}

The auxiliary TASEP is essentially the same as the usual multispecies TASEP; the only difference (besides the fact that particles are called \emph{stones} and species are called \emph{densities}) is that the stones of a given density are given different names. However, the definition of $\Omm$ ensures that this minor difference does not affect the dynamics of the Markov chain. In particular, the auxiliary TASEP has a stationary measure $\omega$ (not necessarily a probability measure) given by 
\begin{equation}\label{eq:omega_circ}\omega(\sigma)=F_{(\rrho(\s_{\sigma^{-1}(1)}),\ldots,\rrho(\s_{\sigma^{-1}(n)}))}(1,\ldots,1;0). 
\end{equation}

Let $\ttt=(\chi_1,\ldots,\chi_n)$ be an $n$-tuple of nonzero real numbers such that for all $j,j'\in[n]$ satisfying $\rrho(\s_{j})<\rrho(\s_{j'})$, the number 
\begin{equation}\label{eq:p(jj')}
p(j,j')\coloneq\frac{\chi_j-\chi_{j'}}{t\chi_j-\chi_{j'}}
\end{equation} belongs to $[0,1)$. Let us also assume that there exist $j,j'\in[n]$ such that $p(j,j')>0$. Recall that for $w\in\mathfrak S_n$, we write $w\ttt=(\chi_{w^{-1}(1)},\ldots,\chi_{w^{-1}(n)})$. We envision an element of $S_\lambda\times\Omm$ as a configuration of particles and stones on the ring with sites $1,\ldots,n$. 

\begin{definition}\label{def:stonedmASEP}
The \dfn{stoned multispecies ASEP}, which we denote by $\blacktriangle\AmASEP_\lambda$, is the discrete-time Markov chain with state space $S_\lambda\times\Omm$ whose transition probabilities are as follows: 
\begin{itemize}
\item If $\rrho(\s_{\sigma^{-1}(i)})<\rrho(\s_{\sigma^{-1}(i+1)})$ and $\mu_i\neq\mu_{i+1}$, then \[\mathbb P((\mu,\sigma)\to (\overline{s}_i\mu,\overline{s}_i\sigma))=\frac{1}{n}p(\sigma^{-1}(i),\sigma^{-1}(i+1))\hh_t(\mu_{i},\mu_{i+1})\] and \[\mathbb P((\mu,\sigma)\to (\mu,\overline{s}_i\sigma))=\frac{1}{n}[1-p(\sigma^{-1}(i),\sigma^{-1}(i+1))\hh_t(\mu_{i},\mu_{i+1})].\]
\item If $\rrho(\s_{\sigma^{-1}(i)})<\rrho(\s_{\sigma^{-1}(i+1)})$ and $\mu_i=\mu_{i+1}$, then \[\mathbb P((\mu,\sigma)\to (\mu,\overline s_i\sigma))=\frac{1}{n}.\]  
\item We have $\mathbb P((\mu,\sigma)\to(\mu,\sigma))=1-\frac{\MM(\sigma)}{n}$. 
\item All other transition probabilities are $0$. 
\end{itemize}
\end{definition}

\begin{figure}[ht]
\begin{center}{\includegraphics[height=7.359cm]{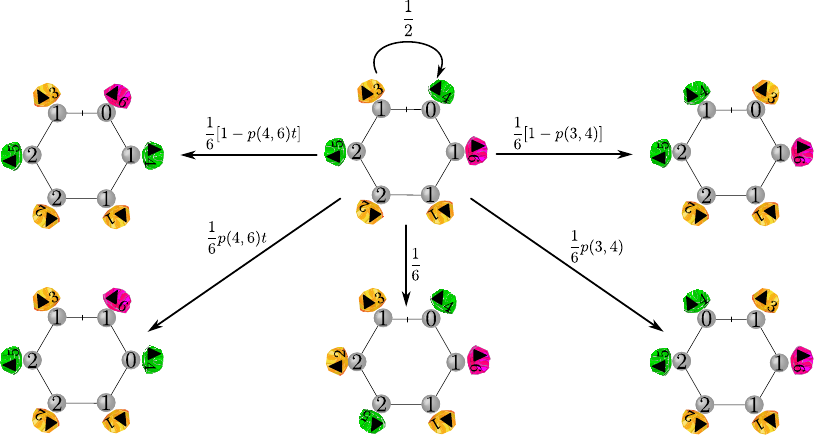}}
\end{center}
\caption{All transitions from the state $((0,1,1,2,2,1),346152)$ of the Markov chain $\blacktriangle\AmASEP_{(0,1,1,1,2,2)}$ with $\rrho(\s_1)=\rrho(\s_2)=\rrho(\s_3)=1$, 
$\rrho(\s_4)=\rrho(\s_5)=2$, and $\rrho(\s_6)=3$. } 
\label{fig:stonedmASEP}
\end{figure}

\cref{fig:stonedmASEP} illustrates \cref{def:stonedmASEP}. 

A more intuitive description of the dynamics of $\blacktriangle\AmASEP_\lambda$ is as follows. The stones move according to the auxiliary TASEP (so the dynamics of the stones do not depend on the particles). If $\blacktriangle\AmASEP_\lambda$ is in a state $(\mu,\sigma)$ and two stones $\s_j$ and $\s_{j'}$ with $j<j'$ swap places (which necessarily implies that $\sigma(j')\equiv\sigma(j)+1\pmod{n}$), then these stones send a signal to the particles on sites $\sigma(j)$ and $\sigma(j')$ telling them to swap places. However, the signal only has probability $p(j,j')$ of actually reaching the particles. If the signal does not reach the particles, then the particles simply do not move. On the other hand, if the particles do receive the signal, then with probability $\hh_t(\mu_{\sigma(j)},\mu_{\sigma(j')})$, they decide to actually follow their orders and swap places (and with probability $1-\hh_t(\mu_{\sigma(j)},\mu_{\sigma(j')})$, they disregard the signal and do not move).  

\subsection{Stationary Distribution} 

The assumption that the probabilities $p(j,j')$ in \eqref{eq:p(jj')} are strictly less than $1$ and are not all $0$ ensures that $\blacktriangle\AmASEP_\lambda$ is irreducible and, hence, has a unique stationary distribution. Our main theorem in this section determines this stationary distribution explicitly in terms of ASEP polynomials. 

\begin{theorem}\label{thm:main_ASEP}
The stationary probability measure $\pi$ of $\blacktriangle\AmASEP_\lambda$ is given by \[\pi(\mu,\sigma)=\frac{1}{Z(\lambda,\rrho)}F_{\mu}(\sigma\ttt;t)\cdot F_{(\rrho(\s_{\sigma^{-1}(1)}),\ldots,\rrho(\s_{\sigma^{-1}(n)}))}(1,\ldots,1;0),\]
where $Z(\lambda,\rrho)$ is a normalization factor that only depends on $\lambda$ and $\rrho$. 
\end{theorem}

If there are only $m=2$ different stone densities, then it is well known that the stationary distribution of the auxiliary TASEP is the uniform distribution on $\Omm$. In other words, the quantity \[\omega(\sigma)=F_{(\rrho(\s_{\sigma^{-1}(1)}),\ldots,\rrho(\s_{\sigma^{-1}(n)}))}(1,\ldots,1;0)\] is independent of $\sigma$ when $m=2$. Therefore, the following corollary is immediate from \cref{thm:main_ASEP}. 

\begin{corollary}\label{cor:main_ASEP}
If there are $m=2$ different densities of stones, then the stationary probability measure $\pi$ of $\blacktriangle\AmASEP_\lambda$ is given by \[\pi(\mu,\sigma)=\frac{1}{Z'(\lambda,\rrho)}F_{\mu}(\sigma\ttt;t),\]
where $Z'(\lambda,\rrho)$ is a normalization factor that only depends on $\lambda$ and $\rrho$.  
\end{corollary}

When $m=2$, \cref{cor:main_ASEP} demonstrates that the stoned multispecies ASEP is a Markov chain whose stationary distribution is determined by ASEP polynomials in which $x_1,\ldots,x_n$ and $t$ are evaluated at generic values. 

\begin{example}\label{exam:1stone}
One particularly simple instance of the stoned multispecies ASEP that we wish to highlight is that in which $m=2$ and ${\rrho(\s_1)=\cdots=\rrho(\s_{n-1})=1}<2=\rrho(\s_{n})$. In this case, the $n(n-1)$ permutations in $\Omm$ correspond to the ways of arranging the $n$ stones on the ring so that the stones other than $\s_n$ appear in the cyclic clockwise order $\s_1,\ldots,\s_{n-1}$. The dynamics of the auxiliary TASEP is very simple: at each step, either none of the stones move (this happens with probability $\frac{n-1}{n}$) or the stone $\s_n$ swaps with the stone to its left (this happens with probability $\frac{1}{n}$). Let us choose probabilities $p(1,n),\ldots,p(n-1,n)\in[0,1)$ arbitrarily subject to the condition that they are not all $0$. Let $\chi_n$ be an arbitrary nonzero real number, and for $1\leq j\leq n-1$, define \[\chi_j=\chi_n\frac{1-p(j,n)}{1-p(j,n)t}.\] Then each equation of the form \eqref{eq:p(jj')} holds, so \cref{cor:main_ASEP} tells us that the stationary probability of a state $(\mu,\sigma)$ is \[\frac{1}{Z'(\lambda,\rrho)}F_\mu(\sigma\ttt;t).\] 

An even more special case that we will need later comes from fixing a probability $p\in(0,1)$, setting $t=0$, and setting $\chi_n=(1-p)^{-1}$ and $p(j,n)=p$ for all $1\leq j\leq n-1$. In this case, $\chi_1=\cdots=\chi_{n-1}=1$. If we let $\ttt^{(k)}$ denote the $n$-tuple whose $k$-th entry is $(1-p)^{-1}$ and whose other entries are all $1$, then we find that 
\begin{equation}\label{eq:pi2}
\pi(\mu,\sigma)=\frac{1}{Z'(\lambda,\rrho)}F_\mu(\ttt^{(\sigma(n))};t).
\end{equation} 
\end{example}

\begin{remark}\label{rem:close_to_0}
Suppose the probabilities $p(j,j')$ are all very close to $0$. Then $\chi_1,\ldots,\chi_n$ are very close together, so $\pi(\mu,\sigma)|\Omm|$ is very close to the stationary probability of $\mu$ in $\AmASEP_\lambda$. Intuitively, this is because the signals that the stones send to the particles seldom actually reach the particles. Hence, the sequence of edges of the ring at which signals successfully reach particles is approximately the same as a sequence of edges of the ring chosen independently and uniformly at random.  
\end{remark}

Our proof of \cref{thm:main_ASEP} is actually quite simple, but it requires a couple of lemmas. 

\begin{lemma}\label{lem:rewriting}
For $\nu\in S_\lambda$ and $i\in\ZZ/n\ZZ$, we have 
\[F_\nu({\bf x};t)=\left(1-\frac{x_{i+1}-x_i}{tx_{i+1}-x_i}\hh_t(\nu_{i},\nu_{i+1})\right)\overline{s}_iF_\nu({\bf x};t)+\frac{x_{i+1}-x_i}{tx_{i+1}-x_i}\hh_t(\nu_{i+1},\nu_{i})\overline{s}_iF_{\overline{s}_i\nu}({\bf x};t).\]
\end{lemma} 

\begin{proof}
If $\nu_i=\nu_{i+1}$, then $\hh_t(\nu_i,\nu_{i+1})=\hh_t(\nu_{i+1},\nu_{i})=0$, so the desired identity follows from \eqref{eq:qKZ2}. 

Now assume that $\nu_i<\nu_{i+1}$. In this case, we have $\hh_t(\nu_i,\nu_{i+1})=t$ and $\hh_t(\nu_{i+1},\nu_i)=1$. If we consider \eqref{eq:qKZ1} with the roles of $x_i$ and $x_{i+1}$ reversed and with $\mu=\nu$, we find that \[t\overline{s}_iF_\nu({\bf x};t)-\frac{tx_{i+1}-x_i}{x_{i+1}-x_i}(1-\overline{s}_i)\overline{s}_iF_\nu({\bf x};t)=\overline{s}_iF_{\overline{s}_i\nu}({\bf x};t),\] and this is equivalent to the desired identity. 

Finally, assume that $\nu_i>\nu_{i+1}$. In this case, we have $\hh_t(\nu_i,\nu_{i+1})=1$ and $\hh_t(\nu_{i+1},\nu_i)=t$. If we consider \eqref{eq:qKZ1} with the roles of $x_i$ and $x_{i+1}$ reversed and with $\mu=\overline{s}_i\nu$, we find that
\begin{equation}\label{eq:prelim}
t\overline{s}_iF_{\overline{s}_i\nu}({\bf x};t)-\frac{tx_{i+1}-x_i}{x_{i+1}-x_i}(1-\overline{s}_i)\overline{s}_iF_{\overline{s}_i\nu}({\bf x};t)=\overline{s}_iF_{\nu}({\bf x};t).
\end{equation} 
Using \eqref{eq:qKZ1} and the definition of $T_i$, one can check that ${(1-\overline{s}_i)(F_\nu({\bf x};t)+F_{\overline{s}_i\nu}({\bf x};t))=0}$. Thus, we can substitute $\overline{s}_iF_\nu({\bf x};t)+\overline{s}_iF_{\overline{s}_i\nu}({\bf x};t)-F_\nu({\bf x};t)$ in the place of $F_{\overline{s}_i\nu}({\bf x};t)$ in \eqref{eq:prelim}; rearranging the resulting equation then yields the desired identity. 
\end{proof}

The following lemma is the crux of the proof of \cref{thm:main_ASEP}. In what follows, we write $\mathbb P$ to denote transition probabilities in both the auxiliary TASEP and $\blacktriangle\AmASEP_\lambda$; it should be clear from context which is meant. 

\begin{lemma}\label{lem:main_ASEP}
Let $\sigma,\widehat\sigma\in\Omm$ be such that the transition probability $\mathbb P(\widehat\sigma\to\sigma)$ in the auxiliary TASEP is positive. For each $\mu\in S_\lambda$, we have 
\[\sum_{\widehat\mu\in S_\lambda}\mathbb P((\widehat\mu,\widehat\sigma)\to(\mu,\sigma))F_{\widehat\mu}(\widehat \sigma\ttt;t)=F_\mu(\sigma\ttt;t)\mathbb P(\widehat\sigma\to\sigma).\]
\end{lemma}

\begin{proof}
If $\widehat\sigma=\sigma$, then the desired result is immediate from \cref{def:stonedmASEP}, which tells us that \[\mathbb P((\mu,\sigma)\to(\mu,\sigma))=1-\frac{\MM(\sigma)}{n}=\mathbb P(\sigma\to\sigma)\] and that ${\mathbb P((\widehat\mu,\sigma)\to(\mu,\sigma))=0}$ for all $\widehat\mu\in S_\lambda\setminus\{\mu\}$. 

Now assume $\widehat\sigma\neq\sigma$. Then $\mathbb P(\widehat\sigma\to\sigma)=\frac{1}{n}$, and there exists $i\in\ZZ/n\ZZ$ such that $\sigma=\overline s_i\widehat\sigma$ and $\rrho(\s_{\sigma^{-1}(i+1)})<\rrho(\s_{\sigma^{-1}(i)})$. Also, $\mathbb P((\widehat\mu,\overline{s}_i\sigma)\to(\mu,\sigma))=0$ for all $\widehat\mu\in S_\lambda\setminus\{\mu,\overline{s}_i\mu\}$. We consider two cases. 

\medskip 

\noindent {\bf Case 1.} Assume $\mu_i=\mu_{i+1}$. In this case, we have $\mathbb P((\mu,\overline{s}_i\sigma)\to(\mu,\sigma))=\frac{1}{n}=\mathbb P(\overline{s}_i\sigma\to\sigma)$. The exchange equation \eqref{eq:qKZ2} tells us that $F_\mu=\overline{s}_iF_\mu$, so 
\begin{align*}
\sum_{\widehat\mu\in S_\lambda}\mathbb P((\widehat\mu,\widehat\sigma)\to(\mu,\sigma))F_{\widehat\mu}(\widehat\sigma\ttt;t)&=\mathbb P((\mu,\overline{s}_i\sigma)\to(\mu,\sigma))F_\mu(\overline{s}_i\sigma\ttt;t) \\  &=\mathbb P((\mu,\overline{s}_i\sigma)\to(\mu,\sigma))(\overline{s}_iF_\mu)(\sigma\ttt;t) \\ 
&=F_\mu(\sigma\ttt;t)\mathbb P(\widehat\sigma\to\sigma). 
\end{align*} 

\medskip 

\noindent {\bf Case 2.} Assume $\mu_i\neq\mu_{i+1}$. In this case, we have \[\mathbb P((\mu,\overline{s}_i\sigma)\to(\mu,\sigma))=\frac{1}{n}[1-p(\sigma^{-1}(i+1),\sigma^{-1}(i))\hh_t(\mu_i,\mu_{i+1})]\] and 
\[\mathbb P((\overline{s}_i\mu,\overline{s}_i\sigma)\to(\mu,\sigma))=\frac{1}{n}p(\sigma^{-1}(i+1),\sigma^{-1}(i))\hh_t(\mu_{i+1},\mu_{i}).\] Recall from \eqref{eq:p(jj')} that $p(\sigma^{-1}(i+1),\sigma^{-1}(i))=\dfrac{\chi_{\sigma^{-1}(i+1)}-\chi_{\sigma^{-1}(i)}}{t\chi_{\sigma^{-1}(i+1)}-\chi_{\sigma^{-1}(i)}}$. The desired result is now immediate if we set $\nu=\mu$ and ${\bf x}=\ttt$ in \cref{lem:rewriting}. 
\end{proof}

\begin{proof}[Proof of \cref{thm:main_ASEP}]
Fix a state $(\mu,\sigma)\in S_\lambda\times\Omm$. We will prove that the desired stationary distribution in the statement of the theorem satisfies the Markov chain balance equation at $(\mu,\sigma)$. This amounts to showing that 
\[\sum_{(\widehat\mu,\widehat\sigma)\in S_\lambda\times\Omm}\mathbb P((\widehat\mu,\widehat\sigma)\to(\mu,\sigma))F_{\widehat \mu}(\widehat\sigma\ttt;t)\omega(\widehat\sigma)=F_\mu(\sigma\ttt;t)\omega(\sigma),\] where $\omega$ is as in \eqref{eq:omega_circ}. Note that $\mathbb P((\widehat\mu,\widehat\sigma)\to(\mu,\sigma))$ is $0$ whenever the transition probability $\mathbb P(\widehat\sigma\to\sigma)$ in the auxiliary TASEP is $0$. Hence, it follows from \cref{lem:main_ASEP} that 
\[\sum_{(\widehat\mu,\widehat\sigma)\in S_\lambda\times\Omm}\mathbb P((\widehat\mu,\widehat\sigma)\to(\mu,\sigma))F_\mu(\widehat\sigma\ttt;t)\omega(\widehat\sigma) =F_\mu(\sigma\ttt;t)\sum_{\widehat\sigma\in\Omm}\mathbb P(\widehat\sigma\to\sigma)\omega(\widehat\sigma). \]
Because $\Omm$ is a stationary measure for the auxiliary TASEP, we know that \[\sum_{\widehat\sigma\in\Omm}\mathbb P(\widehat\sigma\to\sigma)\omega(\widehat\sigma)=\omega(\sigma),\] 
which completes the proof. 
\end{proof}

\subsection{Random Combinatorial Billiards}\label{subsec:billiards1} 

In this subsection, we specialize to the setting where $\lambda=(1,2,\ldots,n)$. The map $\SSS_n\to S_\lambda$ given by $w\mapsto(w^{-1}(1),\ldots,w^{-1}(n))$ allows us to identify $\SSS_n$ with $S_\lambda$. Let us also assume that $m=2$ and $\rrho(\s_1)=1<2=\rrho(\s_2)=\cdots=\rrho(\s_n)$. 

Fix a probability $p\in(0,1)$, and let $p(1,j)=p$ for all $2\leq j\leq n$. Let $\ttt=(\chi,1,\ldots,1)$, where \[\chi=\frac{1-p}{1-pt}.\] Then each equation of the form \eqref{eq:p(jj')} holds. Let $\ttt^{(k)}$ denote the $n$-tuple whose $k$-th entry is $\chi$ and whose other entries are all $1$ (so $\ttt=\ttt^{(1)}$). According to \cref{cor:main_ASEP}, the stationary probability of a state $(w,\sigma)$ of $\blacktriangle\AmASEP_\lambda$ is 
\begin{equation}\label{eq:pi1}
\pi(w,\sigma)=\frac{1}{Z'(\lambda,\rrho)}F_w(\ttt^{(\sigma(1))};t).
\end{equation}

Now consider the following random combinatorial billiard trajectory. Start at a point in the interior of the alcove $\BB$, and shine a beam of light in the direction of the vector ${{\ddd^{(n)}}=-ne_n+\sum_{j\in[n]}e_j}$. If at some point in time the beam of light is traveling in an alcove $\BB u$ and hits the hyperplane $\HH^{(u,s_i)}$, then it passes through with probability $p\,\hh_t(\overline{u}(i),\overline{u}(i+1))$ (thereby moving into the alcove $\BB s_iu$), and it reflects with probability $1-p\,\hh_t(\overline{u}(i),\overline{u}(i+1))$ (note that $\hh_t(\overline{u}(i),\overline{u}(i+1))$ only depends on the side of the hyperplane the light beam hits, not the particular alcove $\BB u$). Let us discretize this billiard trajectory; if the beam of light is in the alcove $\BB u$ and it is headed toward the facet of $\BB u$ contained in the hyperplane $\HH^{(u,s_i)}$, then we record the pair $(u,i)\in\affS_n\times\ZZ/n\ZZ$. Applying the quotient map $\affS_n\times\ZZ/n\ZZ\to\SSS_n\times\ZZ/n\ZZ$ defined by $(u,i)\mapsto(\overline u,i)$, we obtain a Markov chain ${\bf M}_{{\ddd^{(n)}}}^{(t)}$ with state space $\SSS_n\times\ZZ/n\ZZ$, which models a certain random combinatorial billiard trajectory in the $(n-1)$-dimensional torus $\mathbb T=V^*/Q^\vee$. Given a state $(w,i)$ of ${\bf M}_{{\ddd^{(n)}}}^{(t)}$, we can encode the permutation $w$ as usual by placing particles of species $w^{-1}(1),\ldots,w^{-1}(n)$ on the sites $1,\ldots,n$ (respectively) of the ring. We can also encode $i$ by placing an unnamed stone of density $1$ on site $i$ and placing unnamed (and indistinguishable) stones of density $2$ on all the other sites. (See \cref{fig:A}.) Let $\zeta_{{\ddd^{(n)}}}^{(t)}$ denote the stationary distribution of ${\bf M}_{{\ddd^{(n)}}}^{(t)}$. 

\begin{figure}[ht]
\begin{center}{\includegraphics[width=\linewidth]{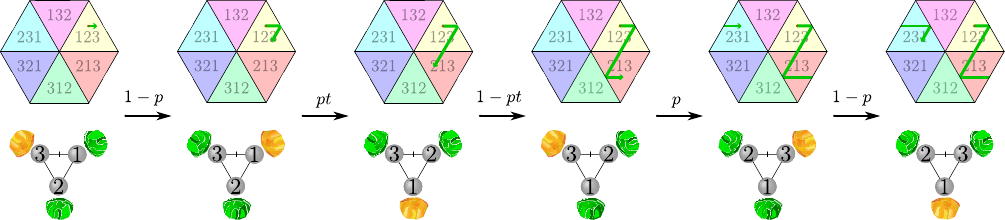}}
\end{center}
\caption{A sequence of transitions in ${\bf M}_{\ddd^{(3)}}^{(t)}$. Each state is represented both as a beam of light traveling in the $2$-dimensional torus $\mathbb T$ (top) and as a configuration of particles and stones on a ring with $3$ sites. The one stone of density $1$ is in {\color{Gold}gold}, while the two stones of density $2$ are in {\color{MildGreen}green}. Each transition is labeled with its probability.}   
\label{fig:A}
\end{figure}  

Our decision to initially shine the beam of light in the direction of the specific vector ${\ddd^{(n)}}$ is important; indeed, the fact that $\mathsf{w}(\ddd^{(n)})=\cdots s_{i_1}s_{i_0}$ with $i_j=j\in\ZZ/n\ZZ$ leads to the following simple description of the dynamics of ${\bf M}_{\ddd^{(n)}}^{(t)}$. In this Markov chain, we have 
\[\mathbb P((w,i)\to(\overline s_iw,i+1))=p\,\hh_t({w}^{-1}(i),{w}^{-1}(i+1))\] and \[\mathbb P((w,i)\to(w,i+1))=1-p\,\hh_t({w}^{-1}(i),{w}^{-1}(i+1)),\]
and all other transition probabilities are $0$. Thus, if the Markov chain is in state $(w,i)$, the stone of density $1$ moves (deterministically) one step clockwise from site $i$ to site $i+1$. When the stone moves, it sends a signal to the particles on sites $i$ and $i+1$ telling them to swap places. The signal reaches the particles with probability $p$. If the signal does not reach the particles, then the particles do not move. If the signal does reach the particles, then with probability $\hh_t(w^{-1}(i),w^{-1}(i+1))$, they decide to follow their orders and swap places, and with probability $1-\hh_t(w^{-1}(i),w^{-1}(i+1))$, they disregard the signal and stay put.  

The Markov chain ${\bf M}_{{\ddd^{(n)}}}^{(t)}$ is very similar to $\blacktriangle\AmASEP_\lambda$, but there are two notable differences. First, the auxiliary TASEP includes laziness. In other words, there are transitions in $\blacktriangle\AmASEP_\lambda$ in which no stones or particles move; this is in contrast to ${\bf M}_{{\ddd^{(n)}}}^{(t)}$, where the stone of density $1$ moves clockwise $1$ space at each time step. Second, in $\blacktriangle\AmASEP_\lambda$, the stones are given different names, while in ${\bf M}_{{\ddd^{(n)}}}^{(t)}$, the stones of density $2$ are indistinguishable. It is straightforward to see that these differences do not affect the stationary distribution except by scaling by a common factor. More precisely, for $w\in\SSS_n$ and $\sigma\in\Omm$, we have
\begin{equation}\label{eq:zeta_pi}\zeta_{{\ddd^{(n)}}}^{(t)}(w,\sigma(1))=(n-1)\pi(w,\sigma)
\end{equation} 
(the factor of $n-1$ is simply due to the fact that $|\Omm|=(n-1)|\ZZ/n\ZZ|$). 

The above discussion and \cref{cor:main_ASEP} tell us that $\zeta_{{\ddd^{(n)}}}^{(t)}$ is determined by ASEP polynomials. More precisely, 
\begin{equation}\label{eq:zeta_dd}
\zeta_{{\ddd^{(n)}}}^{(t)}(w,M)=\frac{1}{n}\frac{F_w(\ttt^{(M)};t)}{P_{(1,2,\ldots,n)}(\ttt;t)}.
\end{equation}

\subsection{Reduced Random Combinatorial Billiards}\label{subsec:reducedA} 

As before, let ${\ddd^{(n)}}=-ne_n+\sum_{j\in[n]}e_j$. The Markov chain ${\bf M}_{\ddd^{(n)}}$ defined in \cref{subsec:intro_billiards} is exactly the same as the Markov chain ${\bf M}_{\ddd^{(n)}}^{(t)}$ defined in the preceding subsection when $t=0$. Therefore, we can set $t=0$ in \eqref{eq:zeta_dd} to find that the vector $\psi_{\ddd^{(n)}}$ in \cref{thm:LamBilliards} is a positive scalar multiple of 
\begin{equation}\label{eq:AyyerIntermediate}
\sum_{\substack{w\in\SSS_n \\ w^{-1}(1)<w^{-1}(n)}}F_w(1,\ldots,1,1-p;0)\,(e_{w^{-1}(1)}-e_{w^{-1}(n)}).
\end{equation} In \cref{sec:correlations}, we will use multiline queues to prove \cref{thm:AyyerLinussonBilliards}, thereby providing an even simpler and more explicit description of $\langle\psi_{\ddd^{(n)}}\rangle$.

Let us now prove \cref{cor:LamBilliards2}. We need to understand the stationary distribution $\zeta_{-{\ddd^{(n)}}}$ of ${\bf M}_{-{\ddd^{(n)}}}$. To this end, we consider the stoned multispecies ASEP with an auxiliary TASEP that is different from the one used in \cref{subsec:billiards1}. Namely, we must now use the stone density function $\rrho$ defined by ${\rrho(\s_1)=\cdots=\rrho(\s_{n-1})=1}$ and $\rrho(\s_n)=2$. In this setting, the stone $\s_n$ moves counterclockwise around the cycle, as in \cref{exam:1stone}. We can encode a state $(w,M)$ of ${\bf M}_{-{\ddd^{(n)}}}$ by placing particles of species $w^{-1}(1),\ldots,w^{-1}(n)$ on the sites $1,\ldots,n$ (respectively) of the ring, placing an unnamed stone of density $2$ on site $n-M$, and placing unnamed (an indistinguishable) stones of density $1$ on all the other sites. Let $\overleftarrow\ttt=(1,\ldots,1,(1-p)^{-1})$, and let $\overleftarrow\ttt^{(k)}$ denote the $n$-tuple whose $k$-th entry is $(1-p)^{-1}$ and whose other entries are all $1$. An argument very similar to the one used to derive \eqref{eq:zeta_dd} (but now setting $t=0$ and using \eqref{eq:pi2} instead of \eqref{eq:pi1}) yields that 
\[\zeta_{-{\ddd^{(n)}}}(w,M)=\frac{1}{n}\frac{F_w(\overleftarrow\ttt^{(n-M)};0)}{P_{(1,2,\ldots,n)}(\overleftarrow\ttt;0)}.\]
By combining this identity and \eqref{eq:zeta_dd} (with $t=0$), we deduce \cref{cor:LamBilliards2} from \cref{thm:LamBilliards2}. 

\section{Correlations and Cores}\label{sec:correlations} 

Throughout this section, we assume $t=0$ unless otherwise stated. 

\subsection{Multiline Queues}\label{subsec:multiline_queues}
Let $m_i$ denote the number of entries equal to $i$ in our fixed $n$-tuple $\lambda$. Consider a $(\lambda_n+1)\times n$ cylindrical array $\mathsf{C}$ of sites, with rows numbered $0,1,\ldots,\lambda_n$ from top to bottom and with columns numbered $1,\ldots,n$ from left to right. The array is cylindrical in the sense that column $n$ is adjacent to column $1$. A \dfn{multiline queue} is an arrangement of balls in some of the sites of this array such that for each $0\leq r\leq \lambda_n$, row $r$ contains exactly $m_0+\cdots+m_r$ balls. Note that the total number of multiline queues is $\prod_{r=0}^{\lambda_n}\binom{n}{m_0+\cdots+m_r}$. We define the \dfn{weight} of a multiline queue $\QQ$ to be the monomial \[\wt_\QQ({\bf x})=x_1^{g_1(\QQ)}\cdots x_n^{g_n(\QQ)},\] where $g_j(\QQ)$ is the number of balls in $\QQ$ that lie in column $j$ and in one of the rows $0,1,\ldots,\lambda_n-1$ (note that this definition ignores balls in the bottom row).  

Let $\QQ$ be a multiline queue. We will construct a sequence $\pp_1,\ldots,\pp_n$, where each $\pp_i$ is a path in $\mathsf{C}$ called a \dfn{bully path}. Let $\ell\in[n]$, and suppose we have already constructed the bully paths $\pp_1,\ldots,\pp_{\ell-1}$. Let $r$ be the unique integer such that $m_0+\cdots+m_{r-1}<\ell\leq m_0+\cdots +m_{r}$. Let $b_\ell(r)$ be the leftmost ball in row $r$ that does not already have a bully path passing through it. Say $b_\ell(r)$ is in column $c_\ell(r)$. If $r<\lambda_n$, let $c_\ell(r+1)$ be the first entry in the list $c_\ell(r),c_\ell(r)+1,\ldots,c_\ell(r)+n-1$ (with entries taken modulo $n$) such that there is a ball in row $r+1$ and column $c_{\ell}(r+1)$ that does not already have a bully path passing through it, and let $b_\ell(r+1)$ be that ball. If $r+1<\lambda_n$,
let $c_\ell(r+2)$ be the first entry in the list $c_\ell(r+1),c_\ell(r+1)+1,\ldots,c_\ell(r+1)+n-1$ such that there is a ball in row $r+2$ and column $c_{\ell}(r+2)$ that does not already have a bully path passing through it, and let $b_\ell(r+2)$ be that ball. In general, if $r+k<\lambda_n$, let $c_\ell(r+k+1)$ be the first entry in the list $c_\ell(r+k),c_\ell(r+k)+1,\ldots,c_\ell(r+k)+n-1$ such that there is a ball in row $r+k+1$ and column $c_{\ell}(r+k+1)$ that does not already have a bully path passing through it, and let $b_\ell(r+k+1)$ be that ball. This process stops once the balls $b_\ell(r),b_\ell(r+1),\ldots,b_\ell(\lambda_n)$ have been defined. We then draw the bully path $\pp_\ell$, which is defined to be a path in $\mathsf{C}$ that moves monotonically down and to the right and that passes through the balls $b_\ell(r),b_\ell(r+1),\ldots,b_\ell(\lambda_n)$ and no other balls. See \cref{fig:multiline1}. 

\begin{figure}[ht]
\begin{center}{\includegraphics[height=4.491cm]{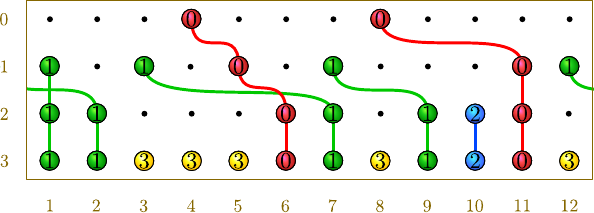}}
\end{center}
\caption{A multiline queue $\QQ$ for $\lambda=(0,0,1,1,1,1,2,3,3,3,3,3)$. We have $\bp(\QQ)=(1,1,3,3,3,0,1,3,1,2,0,3)$. }
\label{fig:multiline1}
\end{figure}

Once we have drawn all the bully paths $\pp_1,\ldots,\pp_n$ for $\QQ$, each ball will have a unique bully path passing through it. If a ball $b$ has the path $\pp_\ell$ passing through it, then we label $b$ with the unique number $r$ such that $m_0+\cdots+m_{r-1}<\ell\leq m_0+\cdots +m_{r}$ (that is, $r$ is the number of the row where the bully path passing through $b$ starts). By reading the labels of the balls in row $\lambda_n$ from left to right, we obtain a tuple $\bp(\QQ)\in S_\lambda$. For $\mu\in S_\lambda$, let $\mathfrak{Q}(\mu)$ be the set of multiline queues $\QQ$ such that $\bp(\QQ)=\mu$. A consequence of the main result from \cite{CMW} is that the ASEP polynomial $F_\mu$ (with $t=0$) is given by
\begin{equation}\label{eq:multiline}
F_\mu({\bf x};0)=\sum_{\QQ\in\mathfrak{Q}(\mu)}\wt_\QQ({\bf x}). 
\end{equation}

\subsection{Correlations} 
As in \cref{subsec:billiards1}, we now specialize to the setting where $\lambda=(1,2,\ldots,n)$, and we use the map $w\mapsto(w^{-1}(1),\ldots,w^{-1}(n))$ to identify $\SSS_n$ with $S_\lambda$. We also assume that $\rrho(\s_1)=1$ and $\rrho(\s_2)=\cdots=\rrho(\s_n)=2$. 

Given a nondecreasing $n$-tuple of nonnegative integers $\lambda'$, we define a map $\proj_{\lambda'}\colon\SSS_n\to S_{\lambda'}$ by \[\proj_{\lambda'}(w)=(\lambda'_{w^{-1}(1)},\ldots,\lambda'_{w^{-1}(n)}).\] Let $\pi_\lambda$ and $\pi_{\lambda'}$ denote the stationary distributions of $\blacktriangle\AmASEP_\lambda$ and $\blacktriangle\AmASEP_{\lambda'}$, respectively (defined using the same auxiliary TASEP). It is straightforward to check that 
\[\pi_{\lambda'}(\mu',\sigma)=\sum_{\substack{\mu\in S_\lambda \\ \proj_{\lambda'}(\mu)=\mu'}}\pi_\lambda(\mu,\sigma) 
\]
for all $(\mu',\sigma)\in S_{\lambda'}\times\Omega$. This equation encodes the \emph{projection principle}, which allows us to reduce problems concerning $\blacktriangle\AmASEP_{\lambda'}$ (for arbitrary $\lambda'$) to problems concerning $\blacktriangle\AmASEP_{\lambda}$ (for $\lambda=(1,2,\ldots,n)$). Thus, while the results in this section concern correlations in $\blacktriangle\AmASEP_{\lambda}$, one could use the projection principle to obtain similar results for $\blacktriangle\AmASEP_{\lambda'}$. 

Fix a permutation $\sigma\in\Omega$ such that $\sigma(1)=n$. 
For nonnegative integers $k$ and $\ell$ such that $k+\ell\leq n$, let 
\[\lambda(k,\ell)=(\underbrace{0,\ldots,0}_{k},\underbrace{1,\ldots,1}_{\ell},\underbrace{2\ldots,2}_{n-k-\ell})\] be the nondecreasing $n$-tuple that has $k$ copies of $0$, has $\ell$ copies of $1$, and has $n-k-\ell$ copies of $2$.

For integers $i$ and $j$ such that $1\leq i<j\leq n$, we wish to understand the correlation between the positions of the particles of species $i$ and $j$ in $\blacktriangle\AmASEP_\lambda$ when the auxiliary TASEP is in the state $\sigma$. More precisely, we are interested in the stationary probability of the event that the particle of species $i$ is immediately to the right of the particle with species $j$. By rotational symmetry, we may focus on the case in which the particles sit on sites $n$ and $1$. Thus, we wish to understand the quantity 
\[E_{j,i}=E_{j,i}^\sigma=\sum_{\substack{\mu\in S_\lambda \\ \mu_1=i,\,\mu_n=j}}\pi_\lambda(\mu,\sigma).\] 
It follows from \cref{thm:LamBilliards} and \eqref{eq:zeta_pi} (with $t=0$) that $\psi_{\ddd^{(n)}}$ is a positive scalar multiple of \[\sum_{1\leq i<j\leq n}E_{j,i}(e_i-e_j).\] Thus, the following proposition implies \cref{thm:AyyerLinussonBilliards}. 

\begin{proposition}\label{prop:correlations}
Fix $\sigma\in\Omm$ with $\sigma(1)=n$. For $1\leq i<j\leq n$, the quantity $E_{j,i}=E_{j,i}^\sigma$ is independent of $\sigma$ and is given by 
\[E_{j,i}=\frac{n(1-p)}{n-1}\cdot\frac{(j-i)(2n-(i+j-1)p)}{(n-ip)(n-(i-1)p)(n-jp)(n-(j-1)p)}.\]
\end{proposition} 

\begin{proof}
We follow the approach of Ayyer and Linusson from \cite{AyyerLinusson}. Define \[D_{k,\ell}=\sum_{\substack{\mu\in S_{\lambda(k,\ell)} \\ \mu_1=1,\,\mu_n=2}}\pi_{\lambda(k,\ell)}(\mu,\sigma).\] By the projection principle,  
\[D_{k,\ell}=\sum_{k<i\leq k+\ell<j\leq n}E_{j,i}.\] By the principle of inclusion-exclusion, we have 
\begin{equation}\label{eq:PIE}
E_{j,i}=D_{i-1,j-i}-D_{i,j-i-1}-D_{i-1,j-i+1}+D_{i,j-i}. 
\end{equation} 
Thus, we just need to know how to compute $D_{k,\ell}$; this is manageable because it just requires us to understand a stoned $3$-species ASEP (rather than a stoned $n$-species ASEP).  

Suppose $\mu\in S_{\lambda(k,\ell)}$ is such that $\mu_1=1$ and $\mu_n=2$. Observe that every multiline queue $\QQ\in \mathfrak{Q}(\mu)$ must have the following form: 
\[\begin{array}{l}\includegraphics[height=3.48cm]{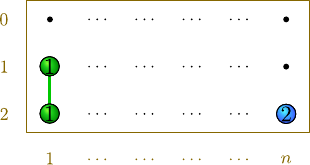}\end{array}.\]
That is, among the four positions in rows $0$ and $1$ and columns $1$ and $n$, there must be exactly one ball, and that ball must be in row $1$ and column $1$. This implies that every monomial $\wt_\QQ({\bf x})$ for $\QQ\in\mathfrak Q(\mu)$ is independent of $x_n$ (this is the crucial fact that permits the rest of our analysis in this section). It follows from \eqref{eq:multiline} that 
\[\sum_{\substack{\mu\in S_{\lambda(k,\ell)} \\ \mu_1=1,\, \mu_n=2}}F_\mu(1,\ldots,1,1-p;0)=\sum_{\substack{\mu\in S_{\lambda(k,\ell)} \\ \mu_1=1,\, \mu_n=2}}F_\mu(1,\ldots,1,1;0),\] and Ayyer and Linusson already computed that  
\[\sum_{\substack{\mu\in S_{\lambda(k,\ell)} \\ \mu_1=1,\, \mu_n=2}}F_\mu(1,\ldots,1,1;0)=\frac{\ell}{n-k-1}\binom{n-2}{n-k-2}\binom{n-1}{n-k-\ell-1}\] 
(see the proof of \cite[Theorem~4.2]{AyyerLinusson}). Applying \cref{cor:main_ASEP}, we find that 
\begin{equation}\label{eq:Dkl}
D_{k,\ell}=\frac{\ell}{n-k-1}\binom{n-2}{n-k-2}\binom{n-1}{n-k-\ell-1}\left(\sum_{\mu\in S_{\lambda(k,\ell)}}F_\mu(1,\ldots,1,1-p;0)\right)^{-1}.
\end{equation} 

According to \eqref{eq:multiline}, \[\sum_{\mu\in S_{\lambda(k,\ell)}}F_\mu(1,\ldots,1,1-p;0)\] can be computed as a weighted sum over $3\times n$ multiline queues with $k$ balls in row $0$ and $k+\ell$ balls in row $1$, where each multiline queue $\QQ$ is weighted by $(1-p)^{g_n(\QQ)}$. It follows that
\begin{align*}
\sum_{\mu\in S_{\lambda(k,\ell)}}F_\mu(1,\ldots,1,1-p;0)&=\left[\binom{n-1}{k}+\binom{n-1}{k-1}(1-p)\right]\cdot \left[\binom{n-1}{k+\ell}+\binom{n-1}{k+\ell-1}(1-p)\right].
\end{align*} 
Substituting this into \eqref{eq:Dkl} and simplifying, we find that 
\[D_{k,\ell}=\frac{\ell(n-k)(n-k-\ell)}{(n-1)(n-pk)(n-p(k+\ell))}.\] Finally, substituting this formula for $D_{k,\ell}$ into \eqref{eq:PIE} and simplifying the resulting expression for $E_{j,i}$ completes the proof. 
\end{proof}

\subsection{Cores} 

A \dfn{partition} is a nonincreasing tuple of positive integers. We identify a partition $\nu=(\nu_1,\ldots,\nu_\ell)$ with its \dfn{Young diagram}, which is the left-justified arrangement of unit-length axis-parallel boxes in $\mathbb R^2$ in which the $i$-th row (counted from the top) contains $\nu_i$ boxes and the top left corner of the top left box is the origin. For $d\in\ZZ$, we define the $d$-th \dfn{diagonal} to be the line $\{(x,y)\in\mathbb R^2:x+y=d\}$, and we say a unit-length box lies on the $d$-th diagonal if its center is on the $d$-th diagonal. We let ${\bf D}(\nu)\subseteq\mathbb R^2$ denote the scaled version of the Young diagram of $\nu$ whose total area is $1$ and whose upper-left corner is the origin. 

The \dfn{hook} of a box $b$ in $\nu$ is the collection of boxes in $\nu$ that are in the same row as $b$ and lie weakly to the right of $b$ or that are in the same column as $b$ and lie weakly below $b$. The number of boxes in the hook of $b$ is called the \dfn{hook length} of $b$. We say $\nu$ is an \dfn{$n$-core} if none of the hook lengths of the boxes in $\nu$ are divisible by $n$. Let $\Xi_n$ denote the set of $n$-cores.

Let $\nu$ be a partition. Let $\ADD(\nu)$ be the set of boxes $b$ that are not in $\nu$ such that adding $b$ to $\nu$ results in a valid partition. For $d_0\in\ZZ/n\ZZ$, let $\ADD_{d_0}(\nu)$ be the set of boxes $b\in \ADD(\nu)$ such that $b$ lies on a diagonal $d$ with $d\equiv d_0\pmod n$. If $\nu$ is an $n$-core and $\ADD_{d_0}(\nu)$ is nonempty, then we can simultaneously add all boxes in $\ADD_{d_0}(\nu)$ to $\nu$ to obtain a new $n$-core $\gamma_{d_0}(\nu)$; let us draw an arrow labeled $d_0$ from $\nu$ to $\gamma_{d_0}(\nu)$. By drawing all arrows of this form, we obtain an edge-labeled directed graph whose vertex set is the set $\Xi_n$ of $n$-cores; we call this the \dfn{$n$-core graph}. 

Recall that $\widehat\SSS_n=\{w\in\affS_n:\BB w\subseteq\CC\}$ is the set of affine Grassmannian elements of $\affS_n$. For $w\in\widehat\SSS_n$, let $\kappa(w)=\gamma_{j_k}\cdots\gamma_{j_1}(\varnothing)$, where $s_{j_k}\cdots s_{j_1}$ is a reduced word for $w$ and $\varnothing$ is the empty partition. The resulting (well-defined) map $\kappa\colon\widehat\SSS_n\to\Xi_n$ is a bijection. The $n$-core graph and the bijection $\kappa$ are illustrated in \cref{fig:B} when $n=3$. 

\begin{figure}[ht]
\begin{center}{\includegraphics[height=8cm]{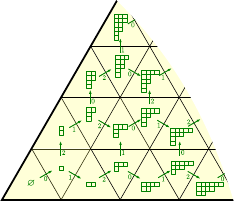}}
\end{center}
\caption{Part of the $3$-core graph drawn inside the fundamental chamber of $\mathcal H_{\SSS_3}$. Each $3$-core $\nu$ is drawn inside the alcove $\BB\kappa^{-1}(\nu)$.} 
\label{fig:B}
\end{figure} 

By passing the affine Grassmannian reduced random walk through the bijection $\kappa$, Lam formulated a new random growth process for $n$-cores \cite{Lam}. In a similar vein, we can pass our affine Grassmannian reduced random combinatorial billiard trajectories through $\kappa$ to obtain a different random growth process. In other words, if we let $(v_M,M)$ denote the state of $\widehat{\bf M}_{\ddd^{(n)}}$ at time $M$, then $(\kappa(v_M))_{M\geq 0}$ is a (stochastically) growing sequence of $n$-cores. 

Given points ${\bf r},{\bf r}'\in\mathbb R^2$, let $[{\bf r},{\bf r}']$ denote the line segment whose endpoints are ${\bf r}$ and ${\bf r}'$. Let $\dH(R_1,R_2)$ denote the Hausdorff distance between two sets $R_1,R_2\subseteq\mathbb R^2$. The next proposition is a consequence of \cite[Proposition~8.10]{LLMS}. 

\begin{proposition}\label{prop:limit_shape}
Let $\beta=(\beta_1,\ldots,\beta_n)\in \CC\setminus\{0\}$, and set $h(i)=\beta_i-\beta_{i+1}$. Let \[{\bf r}_k=\left(\sum_{i=1}^{k-1}ih(i),-\sum_{i=k}^{n-1}(n-i)h(i)\right)\in\mathbb R^2.\] Let $\mathring{\bf R}^{(p)}_\beta$ be the region in $\mathbb R^2$ bounded by the coordinate axes and the curve $\bigcup_{k=1}^{n-1}[{\bf r}_k,{\bf r}_{k+1}]$. Let ${\bf R}^{(p)}_\beta=\mathrm{A}_\beta^{-1/2}\mathring{\bf R}^{(p)}_\beta$, where $\mathrm{A}_\beta$ is the area of $\mathring{\bf R}^{(p)}_\beta$. If $(w_M)_{M\geq 0}$ is a sequence of affine Grassmannian elements of $\affS_n$ such that 
\[\lim_{M\to\infty}\langle \BB w_M\rangle=\langle\beta\rangle,\] then 
\[\lim_{M\to\infty}\dH({\bf D}(\kappa(w_M)),{\bf R}^{(p)}_\beta)=0.\]  
\end{proposition}

By combining \cref{prop:limit_shape} with \cref{thm:AyyerLinusson}, Ayyer and Linusson gave an exact description of the limit shape of the scaled Young diagrams in Lam's random growth process for $n$-cores (see \cite[Theorem~3.2]{AyyerLinusson}). We will carry out a similar (though a bit more involved) analysis for our random growth process $(\kappa(v_M))_{M\geq 0}$ (where $(v_M,M)$ is the state of $\widehat{\bf M}_{\ddd^{(n)}}$ at time $M$). 

Let 
$\gamma^{(n)}=(\gamma^{(n)}_1,\ldots,\gamma^{(n)}_n)=\sum_{1\leq i<j\leq n}\widehat E_{j,i}(e_i-e_j)$, where
\[\widehat E_{j,i}=\frac{(j-i)(2n-(i+j-1)p)}{(n-ip)(n-(i-1)p)(n-jp)(n-(j-1)p)},\]
and set $h(i)=\gamma^{(n)}_i-\gamma^{(n)}_{i+1}$. Then 
\begin{align}
\nonumber h(i)&=\sum_{j=1}^n\left(\widehat E_{j,i}-\widehat E_{j,i+1}\right) \\ 
\nonumber &= \sum_{j=1}^n\left(\frac{2}{(n-(i-1)p)(n-ip)(n-(i+1)p)}\right) \\ 
\label{eq:h(i)}&=\frac{2n}{(n-(i-1)p)(n-ip)(n-(i+1)p)}. 
\end{align}
Substituting this formula for $h(i)$ into \cref{prop:limit_shape} allows one to compute the region ${\bf R}^{(p)}_{\gamma^{(n)}}$. \cref{thm:AyyerLinussonBilliards,prop:limit_shape} then combine to yield the following result, which is illustrated (for $n=3$ and $p=3/4$) in \cref{fig:limit_shape}. 

\begin{corollary}\label{cor:shape}
Let $(v_M,M)$ be the state of $\widehat{\bf M}_{\ddd^{(n)}}$ at time $M$. Then 
\[\lim_{M\to\infty}\dH({\bf D}(\kappa(v_M)),{\bf R}^{(p)}_{\gamma^{(n)}})=0.\]   
\end{corollary}

\begin{figure}[ht]
\begin{center}{\includegraphics[height=6cm]{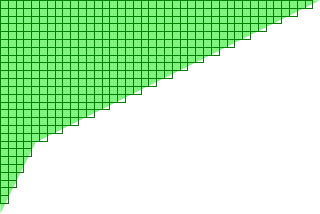}}
\end{center}
\caption{In {\color{DarkGreen}dark green} is the $3$-core corresponding to the last visible point in the affine Grassmannian reduced random billiard trajectory shown in \cref{fig:Trajectory2} (with $p=3/4$). In {\color{LightGreen}light green} is the limit shape ${\bf R}^{(p)}_{\gamma^{(3)}}$, which is bounded by the coordinate axes and the line segments $[(0,-6\sqrt{62/3}),(\sqrt{62/3},-4\sqrt{62/3})]$ and $[(\sqrt{62/3},-4\sqrt{62/3}),(9\sqrt{62/3},0)]$.} 
\label{fig:limit_shape}
\end{figure} 

The next proposition computes the limit of the regions ${\bf R}^{(p)}_{\gamma^{(n)}}$ as $n\to\infty$. 

\begin{remark}
One of the anonymous referees pointed out that \cref{prop:shape} can also be deduced by analyzing
the fan solution of the hydrodynamic equations of the TASEP with geometric jumps (see, e.g., equation (63) in \cite{Rajewsky}). 
\end{remark} 

\begin{proposition}\label{prop:shape} 
We have \[\lim_{n\to\infty}\dH({\bf R}^{(p)}_{\gamma^{(n)}},{\bf R}_\infty^{(p)})=0,\] where \[{\bf R}_\infty^{(p)}=\{(x,y)\in\mathbb R^2:y\leq 0\leq x,\, \sqrt{(1-p)x}+\sqrt{-y}\leq (6(1-p))^{1/4}\}.\] 
\end{proposition} 

\begin{proof}
We consider asymptotics as $n\to\infty$. Let $h(i)$ be as in \eqref{eq:h(i)}. Choose $\alpha\in[0,1]$, and let $k=\left\lfloor\alpha n\right\rfloor$. Let ${\bf r}_k$ be as in \cref{prop:limit_shape}. Note that ${\bf r}_k$ is on the boundary of ${\bf R}^{(p)}_{\gamma^{(n)}}$. The first coordinate of ${\bf r}_k$ is 
\begin{align*}
\sum_{i=1}^{k-1}ih(i)&=\sum_{i=1}^{k-1}\frac{2ni}{(n-ip)^3}(1+O(n^{-2})) \\ 
&=\sum_{i=1}^{k-1}\frac{2i/n}{(1-p(i/n))^3}\frac{1}{n}(1+O(n^{-2})) \\ 
&= \int_0^\alpha \frac{2u}{(1-pu)^3}\,\mathrm{d}u\,(1+o(1)) \\ 
&= \frac{\alpha^2}{(1-p\alpha)^2}(1+o(1)), 
\end{align*}
and the second coordinate of ${\bf r}_k$ is 
\begin{align*}
-\sum_{i=k}^{n-1}(n-i)h(i)&=-\sum_{i=k}^{n-1}\frac{2n(n-i)}{(n-ip)^3}(1+O(n^{-2})) \\ 
&=-\sum_{i=k}^{n-1}\frac{2(1-i/n)}{(1-p(i/n))^3}\frac{1}{n}(1+O(n^{-2})) \\ 
&= -\int_\alpha^1 \frac{2(1-u)}{(1-pu)^3}\,\mathrm{d}u\,(1+o(1)) \\ 
&= -\frac{(1-\alpha)^2}{(1-p)(1-p\alpha)^2}(1+o(1)), 
\end{align*}
This shows that ${\bf r}_k$, which is on the boundary of the region $\mathring{\bf R}^{(p)}_{\gamma^{(n)}}$, is within an asymptotically vanishing distance of the point $(\frac{\alpha^2}{(1-p\alpha)^2},-\frac{(1-\alpha)^2}{(1-p)(1-p\alpha)^2})$, which lies on the curve 
\begin{equation}\label{eq:curve}
\{(x,y)\in\mathbb R^2:\sqrt{(1-p)x}+\sqrt{-y}=(1-p)^{-1/2}\}.
\end{equation} 
It follows that as $n\to\infty$, the regions $\mathring{\bf R}^{(p)}_{\gamma^{(n)}}$ converge (in the Hausdorff distance) to the region $\mathring{\bf R}_\infty^{(p)}$ bounded by the coordinate axes and the curve in \eqref{eq:curve}. The limit region ${\bf R}_\infty^{(p)}$ is then obtained by rescaling $\mathring{\bf R}_\infty^{(p)}$ to have area $1$. 
\end{proof}

\section{The Stoned Inhomogeneous TASEP}\label{sec:inhomogeneous} 

\subsection{Set-up}
We continue to assume that $W=\SSS_n$ and $\widetilde W=\affS_n$. We keep the same conventions regarding the action of $\SSS_n$ on $n$-tuples and polynomials as in \cref{sec:ring}. As before, let $c$ be the permutation with cycle decomposition $(1\,\,2\,\cdots\,n)$. 

In this section, we fix a parameter $a_k\in(0,1)$ for each nonnegative integer $k$. The \dfn{inhomogeneous TASEP} is the discrete-time Markov chain $\iTASEP_\lambda$ with state space $S_\lambda$ in which the transition probability from a state $\mu$ to a state $\mu'$ is given by 
\[\mathbb P(\mu\to\mu')=\begin{cases} \frac{1}{n}a_{\mu_i} & \mbox{if }\mu'=\overline{s}_i\mu\text{ and }\mu_i>\mu_{i+1}; \\  1-\sum_{\nu\in S_\lambda\setminus\{\mu\}}\mathbb P(\mu\to\nu) & \mbox{if }\mu=\mu'; \\ 0 & \mbox{otherwise}.  \end{cases}
\]
\cref{fig:iTASEP} illustrates some transitions in $\iTASEP_{(1,2,3,4,5,6,7,8)}$. 

Lam and Williams \cite{LamWilliams} introduced the inhomogeneous TASEP and posed several intriguing conjectures about it, including one stating that the stationary probabilities of $\iTASEP_{(1,2,\ldots,n)}$ can be expressed (up to a normalization factor) as nonnegative integral sums of Schubert polynomials in the parameters $a_1,\ldots,a_n$.
These conjectures have inspired a flurry of work on the inhomogeneous TASEP \cite{AasLinusson, AasSjostrand, AritaMallick, AyyerLinusson2, Cantini, KimWilliams}.  (Cantini \cite{Cantini} also introduced a more general version of the inhomogeneous TASEP that we will not consider here.) 

\begin{figure}[ht]
\begin{center}{\includegraphics[height=6.691cm]{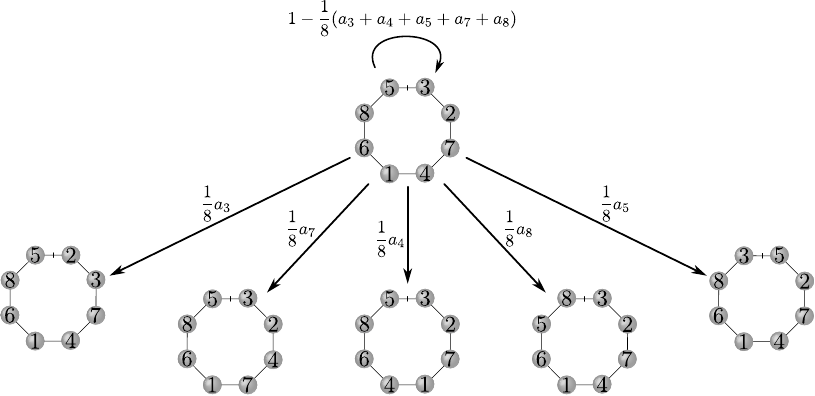}}
\end{center}
\caption{All transitions from the state $(3,2,7,4,1,6,8,5)$ in the Markov chain $\iTASEP_{(1,2,3,4,5,6,7,8)}$. In each depiction of the ring, a tick-mark is drawn between sites $8$ and $1$.} 
\label{fig:iTASEP}
\end{figure} 

In \cite{Cantini}, Cantini introduced a certain family $(\Psi_\mu({\bf x}))_{\mu\in S_\lambda}$ of polynomials in $x_1,\ldots,x_n$ whose coefficients are rational functions in the parameters $a_{\lambda_1},\ldots,a_{\lambda_n}$. These polynomials, which we call \dfn{inhomogeneous TASEP polynomials}, are uniquely determined up to simultaneous multiplication by a symmetric polynomial by the following \dfn{exchange relations}: 
\begin{alignat}{2}
\label{eq:psi1}\frac{x_i(x_{i+1}-a_{\mu_{i+1}})}{a_{\mu_{i+1}}(x_i-x_{i+1})}(1-\overline s_i)\Psi_\mu({\bf x})&=\Psi_{\overline s_i\mu}({\bf x})\quad&&\text{if }i\in\ZZ/n\ZZ\text{ and }\mu_i<\mu_{i+1}; \\
\label{eq:psi2}\overline s_i\Psi_\mu({\bf x})&=\Psi_\mu({\bf x})\quad&&\text{if }i\in\ZZ/n\ZZ\text{ and }\mu_i=\mu_{i+1}; \\ 
\label{eq:psi3}c\Psi_{c\mu}({\bf x})&=\Psi_{\mu}({\bf x}).
\end{alignat} 
Let us fix a normalization of the polynomials $\Psi_\mu({\bf x})$ and define \[\Pi_\lambda({\bf x})=\sum_{\mu\in S_\lambda}\Psi_\mu({\bf x}).\] It is straightforward to check that $(1-\overline s_i)(\Psi_\mu({\bf x})+\Psi_{\overline s_i\mu}({\bf x}))=0$ for all $\mu\in S_\lambda$ and $i\in\ZZ/n\ZZ$. It follows that $\Pi_\lambda({\bf x})$ is symmetric in the variables $x_1,\ldots,x_n$. In \cite{Cantini}, Cantini showed that inhomogeneous TASEP polynomials provide the stationary distribution of $\iTASEP_\lambda$ when $x_1,\ldots,x_n$ ``tend to $\infty$;'' more precisely, the stationary probability of a state $\mu$ in $\iTASEP_\lambda$ is \[\lim_{R\to\infty}\frac{\Psi_\mu(R,\ldots,R)}{\Pi_\lambda(R,\ldots,R)}.\] 
Building off of previous work of Arita and Mallick \cite{AritaMallick} and of Ayyer and Linusson \cite{AyyerLinusson2}, Kim and Williams \cite{KimWilliams} found combinatorial formulas for the polynomials $\Psi_\mu({\bf x})$ in terms of multiline queues. 

The stoned version of $\iTASEP_\lambda$ that we will introduce in this section has a stationary distribution determined by the family $(\Psi_\mu({\bf x}))_{\mu\in S_\lambda}$ with generic choices for all but one of the variables in the tuple ${\bf x}$ (at least one variable must be sent to $\infty$).  

As in \cref{sec:ring}, we consider the auxiliary TASEP with stones $\s_1,\ldots,\s_n$. In order to simplify our approach, we will assume that there are exactly $m=2$ distinct stone densities. As before, let $\MM(\sigma)$ be the number of indices $i\in\ZZ/n\ZZ$ such that $\rrho(\s_{\sigma^{-1}(i)})<\rrho(\s_{\sigma^{-1}(i+1)})$. The transition probabilities in the auxiliary TASEP are again given by \eqref{eq:aux_transitions}.  
One advantage of assuming that there are only $2$ different stone densities is that the stationary distribution of the auxiliary TASEP is the uniform distribution on $\Omm$. In other words,  the quantity $\omega(\sigma)$ in \eqref{eq:omega_circ} does not depend on $\sigma$. 

For each $j\in[n]$ with $\rrho(\s_j)=1$, let us fix a probability 
$p(j)\in[0,1)$. We will assume that there exists $j\in[n]$ such that $\rrho(\s_j)=1$ and $p(j)>0$. Let $R>1$ be a real parameter (which we will imagine is tending to $\infty$). Let $\ttt=(\chi_1,\ldots,\chi_n)$, where 
\[\chi_j=\begin{cases} 1/p(j) & \mbox{if }\rrho(\s_j)=1\text{ and }p(j)>0; \\   R & \mbox{otherwise}.\end{cases}\]

\begin{definition}\label{def:stonediTASEP}
The \dfn{stoned inhomogeneous TASEP}, denoted by $\blacktriangle\iTASEP_\lambda$, is the discrete-time Markov chain with state space $S_\lambda\times\Omm$ whose transition probabilities are as follows: 
\begin{itemize}
\item If $\rrho(\s_{\sigma^{-1}(i)})<\rrho(\s_{\sigma^{-1}(i+1)})$ and $\mu_i>\mu_{i+1}$, then \[\mathbb P((\mu,\sigma)\to (\overline{s}_i\mu,\overline{s}_i\sigma))=\frac{1}{n}p(\sigma^{-1}(i))a_{\mu_{i}}\] and \[\mathbb P((\mu,\sigma)\to (\mu,\overline{s}_i\sigma))=\frac{1}{n}[1-p(\sigma^{-1}(i))a_{\mu_i}].\]
\item If $\rrho(\s_{\sigma^{-1}(i)})<\rrho(\s_{\sigma^{-1}(i+1)})$ and $\mu_i\leq\mu_{i+1}$, then \[\mathbb P((\mu,\sigma)\to (\mu,\overline{s}_i\sigma))=\frac{1}{n}.\]  
\item We have $\mathbb P((\mu,\sigma)\to(\mu,\sigma))=1-\frac{\MM(\sigma)}{n}$. 
\item All other transition probabilities are $0$. 
\end{itemize}
\end{definition}
\cref{fig:D} illustrates \cref{def:stonediTASEP}. 

\begin{figure}[ht]
\begin{center}{\includegraphics[height=7.817cm]{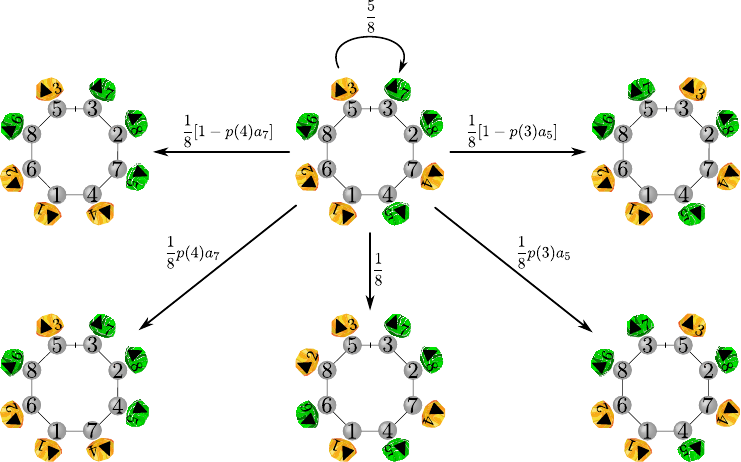}}
\end{center}
\caption{All of the transitions from the state $((3,2,7,4,1,6,8,5),56834712)$ of the Markov chain $\blacktriangle\iTASEP_{(1,2,3,4,5,6,7,8)}$ with $\rrho(\s_1)=\rrho(\s_2)=\rrho(\s_3)=\rrho(\s_4)=1$ and $\rrho(\s_5)=\rrho(\s_6)=\rrho(\s_7)=\rrho(\s_8)=2$.}
\label{fig:D}
\end{figure}

As with the Markov chain $\blacktriangle\AmASEP_\lambda$ from \cref{def:stonedmASEP}, there is an intuitive description of the dynamics of $\blacktriangle\iTASEP_\lambda$. The stones move according to the auxiliary TASEP. If $\blacktriangle\iTASEP_\lambda$ is in a state $(\mu,\sigma)$ and two stones $\s_j$ and $\s_{j'}$ with $j<j'$ swap places (which implies that $\sigma(j')\equiv\sigma(j)+1\pmod{n}$), then these stones send a signal to the particles on sites $\sigma(j)$ and $\sigma(j')$ telling them to swap places. The signal only has probability $p(j)$ of reaching the particles. If the signal does not reach the particles or if $\mu_{\sigma(j)}\leq\mu_{\sigma(j')}$, then the particles do not move. On the other hand, if $\mu_{\sigma(j)}>\mu_{\sigma(j')}$ and the particles receive the signal, then they decide with probability $a_{\mu_{\sigma(j)}}$ to actually swap places (and with probability $1-a_{\mu_{\sigma(j)}}$, they ignore the signal and do not move).  

\subsection{Stationary Distribution} 
Because the probabilities $p(j)$ (for $j\in[n]$ with $\rrho(\s_j)=1$) are strictly less than $1$ and are not all $0$, the Markov chain $\blacktriangle\iTASEP_\lambda$ is irreducible and, hence, has a unique stationary distribution. Our main theorem in this section is the following. 

\begin{theorem}\label{thm:main_iTASEP}
The stationary probability measure $\pi^{\ii}$ of $\blacktriangle\iTASEP_\lambda$ is given by  \[\pi^{\ii}(\mu,\sigma)=\frac{1}{|\Omm|}\lim_{R\to\infty}\frac{\Psi_\mu(\sigma\ttt)}{\Pi_\lambda(\ttt)}.\]
\end{theorem}

\begin{example}\label{exam:1stone_iTASEP}
Suppose $n=6$, $\rrho(\s_1)=\cdots=\rrho(\s_{5})=1$, and $\rrho(\s_6)=2$. Then $|\Omm|=n(n-1)=30$. The dynamics of the auxiliary TASEP are simple: at each step, either none of the stones move or the stone $\s_6$ swaps with the stone to its left. We can choose the probabilities $p(1),\ldots,p(5)\in[0,1)$ arbitrarily so long as they are not all $0$. Say we take $p(1)=p(4)=0$, $p(2)=p(3)=1/3$, and $p(5)=1/7$. Let $\sigma\in\Omm_{\rrho}$ be the permutation whose one-line notation is $612453$. Then \cref{thm:main_iTASEP} tells us that the stationary probability of a state $(\mu,\sigma)$ is \[\frac{1}{30}\lim_{R\to\infty}\frac{\Psi_\mu(\sigma\ttt)}{\Pi_\lambda(\ttt)}=\frac{1}{30}\lim_{R\to\infty}\frac{\Psi_\mu(3,3,R,R,7,R)}{\Pi_\lambda(R,3,3,R,7,R)}.\] 
\end{example} 

\begin{remark}
For $(\mu,\sigma)\in S_\lambda\times\Omm$, note that $\Psi_\mu(\sigma\ttt)$ and $\Pi_\lambda(\ttt)$ are polynomials in $R$ and that the limit in \cref{thm:main_iTASEP} is simply the quotient of their leading coefficients.  
\end{remark}

\begin{remark}\label{rem:close_to_0_i}
If we choose the probabilities $p(j)$ to be very close to $0$, then $\chi_1,\ldots,\chi_n$ will be very large, so $|\Omm|\pi^\ii(\mu,\sigma)$ will be very close to the stationary probability of $\mu$ in the usual $\iTASEP_\lambda$. As in \cref{rem:close_to_0}, this is because the signals that the stones send to the particles rarely reach the particles. 
\end{remark}

In what follows, we write $\mathbb P$ to denote transition probabilities in both the auxiliary TASEP and $\blacktriangle\iTASEP_\lambda$, relying on context to indicate which is meant.

\begin{lemma}\label{lem:main_iTASEP}
Fix $(\mu,\sigma)\in S_\lambda\times\Omm$. Suppose $\widehat\sigma\in\Omm$ is such that the transition probability $\mathbb P(\widehat\sigma\to\sigma)$ in the auxiliary TASEP is positive. There exists a family $(g_{\widehat\mu}^{\widehat\sigma})_{\widehat\mu\in S_\lambda}$ of functions from $\mathbb R_{>1}$ to $\mathbb R$ such that $g_{\widehat\mu}^{\widehat\sigma}(z)=O(1/z)$ and such that 
\[\sum_{\widehat\mu\in S_\lambda}\mathbb P((\widehat\mu,\widehat\sigma)\to(\mu,\sigma))\Psi_{\widehat\mu}(\widehat\sigma\ttt)(1+g_{\widehat\mu}^{\widehat\sigma}(R))=\Psi_\mu(\sigma\ttt)\mathbb P(\widehat\sigma\to\sigma).\]  
\end{lemma}

\begin{proof}
If $\widehat\sigma=\sigma$, then we can simply take $g_{\widehat\mu}^{\widehat\sigma}(z)=0$ for all $\widehat\mu\in S_\lambda$. In this case, the desired result is immediate from \cref{def:stonediTASEP}, which tells us that \[\mathbb P((\mu,\sigma)\to(\mu,\sigma))=1-\frac{\MM(\sigma)}{n}=\mathbb P(\sigma\to\sigma)\] and that ${\mathbb P((\widehat\mu,\sigma)\to(\mu,\sigma))=0}$ for all $\widehat\mu\in S_\lambda\setminus\{\mu\}$. 

Now assume $\widehat\sigma\neq\sigma$. Then $\mathbb P(\widehat\sigma\to\sigma)=\frac{1}{n}$, and there exists $i\in\ZZ/n\ZZ$ such that $\sigma=\overline{s}_i\widehat\sigma$ and $\rrho(\s_{\sigma^{-1}(i)})=2>1=\rrho(\s_{\sigma^{-1}(i+1)})$. We have $\mathbb P((\widehat\mu,\overline{s}_i\sigma)\to(\mu,\sigma))=0$ for all $\widehat\mu\in S_\lambda\setminus\{\mu,\overline{s}_i\mu\}$. We consider two cases. 

\medskip 

\noindent {\bf Case 1.} Assume $\mu_i\geq\mu_{i+1}$. In this case, we have \[\mathbb P((\mu,\overline{s}_i\sigma)\to(\mu,\sigma))=\frac{1}{n}=\mathbb P(\overline{s}_i\sigma\to\sigma)\quad\text{and}\quad\mathbb P((\overline{s}_i\mu,\overline{s}_i\sigma)\to(\mu,\sigma))=0.\] We know by \eqref{eq:psi2} that $\Psi_\mu=\overline{s}_i\Psi_\mu$, so 
\begin{align*}
\sum_{\widehat\mu\in S_\lambda}\mathbb P((\widehat\mu,\widehat\sigma)\to(\mu,\sigma))\Psi_{\widehat\mu}(\widehat\sigma\ttt)&=\mathbb P((\mu,\overline{s}_i\sigma)\to(\mu,\sigma))\Psi_\mu(\overline{s}_i\sigma\ttt) \\  &=\mathbb P((\mu,\overline{s}_i\sigma)\to(\mu,\sigma))(\overline{s}_i\Psi_\mu)(\sigma\ttt) \\ 
&=\Psi_\mu(\sigma\ttt)\mathbb P(\widehat\sigma\to\sigma). 
\end{align*} 
In this case, we can once again take $g_{\widehat\mu}^{\widehat\sigma}(z)=0$ for all $\widehat\mu\in S_\lambda$. 

\medskip 

\noindent {\bf Case 2.} Assume $\mu_i<\mu_{i+1}$. Then \[\mathbb P((\mu,\overline{s}_i\sigma)\to(\mu,\sigma))=\frac{1}{n},\] and 
\[\mathbb P((\overline{s}_i\mu,\overline{s}_i\sigma)\to(\mu,\sigma))=\frac{1}{n}p(\sigma^{-1}(i+1))a_{\mu_{i+1}}.\] 

Suppose first that $p(\sigma^{-1}(i+1))=0$. Then $\chi_{\sigma^{-1}(i)}=\chi_{\sigma^{-1}(i+1)}=R$, so $\sigma\ttt=\overline{s}_i\sigma\ttt$. In this case,
\[\sum_{\widehat\mu\in S_\lambda}\mathbb P((\widehat\mu,\widehat\sigma)\to(\mu,\sigma))\Psi_{\widehat\mu}(\widehat\sigma\ttt)=\frac{1}{n}\Psi_\mu(\overline{s}_i\sigma\ttt)=\frac{1}{n}\Psi_\mu(\sigma\ttt)=\Psi_\mu(\sigma\ttt)\mathbb P(\widehat\sigma\to\sigma),\] 
so we can again take $g_{\widehat\mu}^{\widehat\sigma}(z)=0$ for all $\widehat\mu\in S_\lambda$. 

Now suppose that $p(\sigma^{-1}(i+1))>0$ so that $\chi_{\sigma^{-1}(i+1)}=1/p(\sigma^{-1}(i+1))$. If we set ${\bf x}=\overline{s}_i\sigma\ttt$ in \eqref{eq:psi1}, then we find that 
\begin{align*}
\Psi_\mu(\sigma\ttt)&=\Psi_\mu(\overline{s}_i\sigma\ttt)+\frac{a_{\mu_{i+1}}(\chi_{\sigma^{-1}(i)}-\chi_{\sigma^{-1}(i+1)})}{\chi_{\sigma^{-1}(i+1)}(\chi_{\sigma^{-1}(i)}-a_{\mu_{i+1}})}\Psi_{\overline{s}_i\mu}(\overline{s}_i\sigma\ttt) \\ 
&=\Psi_\mu(\overline{s}_i\sigma\ttt)+a_{\mu_{i+1}}p(\sigma^{-1}(i+1))\frac{R-1/p(\sigma^{-1}(i+1))}{R-a_{\mu_{i+1}}}\Psi_{\overline{s}_i\mu}(\overline{s}_i\sigma\ttt).
\end{align*}
In this case, we take \[g_{\overline{s}_i\mu}^{\widehat\sigma}(z)=\frac{z-1/p(\sigma^{-1}(i+1))}{z-a_{\mu_{i+1}}}-1\] and take $g_{\widehat\mu}^{\widehat\sigma}(z)=0$ for all $\widehat\mu\in S_\lambda\setminus\{\overline{s}_i\mu\}$. (Note that $g_{\overline{s}_i\mu}^{\widehat\sigma}(z)$ is well defined for $z>1$ by our assumption that the parameters $a_k$ are all in $(0,1)$.) 
\end{proof}

\begin{proof}[Proof of \cref{thm:main_iTASEP}]
Fix a state $(\mu,\sigma)\in S_\lambda\times\Omm$. We will prove that the desired stationary distribution in the statement of the theorem satisfies the Markov chain balance equation at $(\mu,\sigma)$. This amounts to showing that 
\begin{equation}\label{eq:desired_iTASEP}
\lim_{R\to\infty}\frac{1}{\Pi_\lambda(\ttt)}\left(\Psi_{\mu}(\sigma\ttt)-\sum_{(\widehat\mu,\widehat\sigma)\in S_\lambda\times\Omm}\mathbb P((\widehat\mu,\widehat\sigma)\to(\mu,\sigma))\Psi_{\widehat\mu}(\widehat\sigma\ttt)\right)=0.
\end{equation}

Let $\text{In}(\sigma)$ denote the set of $\widehat\sigma\in\Omm$ such that the transition probability $\mathbb P(\widehat\sigma\to\sigma)$ in the auxiliary TASEP is positive. Let \[{\bf g}(R)=\max_{\widehat\sigma\in\text{In}(\sigma)}\max_{\widehat\mu\in S_{\lambda}}\left|g_{\widehat\mu}^{\widehat\sigma}(R)\right|,\] where $g_{\widehat\mu}^{\widehat\sigma}$ is as in \cref{lem:main_iTASEP}. Then ${\bf g}(R)=O(1/R)$. We know that \[\sum_{\widehat\sigma\in\Omm}\mathbb P(\widehat\sigma\to\sigma)=1\] because the stationary distribution for the auxiliary TASEP is the uniform distribution on $\Omm$. Also, $\mathbb P((\widehat\mu,\widehat\sigma)\to(\mu,\sigma))=0$ whenever $\widehat\sigma\not\in\text{In}(\sigma)$. Hence, it follows from \cref{lem:main_iTASEP} that 
\begin{align*}
&\hphantom{=}\hspace{0.155cm}\left\lvert\Psi_{\mu}(\sigma\ttt)-\sum_{(\widehat\mu,\widehat\sigma)\in S_\lambda\times\Omm}\mathbb P((\widehat\mu,\widehat\sigma)\to(\mu,\sigma))\Psi_{\widehat\mu}(\widehat\sigma\ttt)\right\rvert \\ &=\left\lvert\sum_{\widehat\sigma\in\text{In}(\sigma)}\left(\Psi_{\mu}(\sigma\ttt)\mathbb P(\widehat\sigma\to\sigma)-\sum_{\widehat\mu\in S_\lambda}\mathbb P((\widehat\mu,\widehat\sigma)\to(\mu,\sigma))\Psi_{\widehat\mu}(\widehat\sigma\ttt)\right)\right\rvert \\ &=\left\lvert\sum_{\widehat\sigma\in\text{In}(\sigma)}g_{\widehat\mu}^{\widehat\sigma}(R)\left(\sum_{\widehat\mu\in S_\lambda}\mathbb P((\widehat\mu,\widehat\sigma)\to(\mu,\sigma))\Psi_{\widehat\mu}(\widehat\sigma\ttt)\right)\right\rvert \\ 
&\leq {\bf g}(R)\sum_{\widehat\sigma\in\text{In}(\sigma)}\sum_{\widehat\mu\in S_\lambda}\Psi_{\widehat\mu}(\widehat\sigma\ttt) \\ 
&={\bf g}(R)|\text{In}(\sigma)|\Pi_\lambda(\ttt) \\ 
&=\Pi_\lambda(\ttt)O(1/R). 
\end{align*}
This implies \eqref{eq:desired_iTASEP}. 
\end{proof} 

\subsection{Random Combinatorial Billiards}\label{subsec:billiards2} 
Assume in this subsection that $\lambda=(1,2,\ldots,n)$, and use the map $w\mapsto(w^{-1}(1),\ldots,w^{-1}(n))$ to identify $\SSS_n$ with $S_\lambda$. Assume that $\rrho(\s_1)=1$ and ${\rrho(\s_2)=\cdots=\rrho(\s_n)=2}$. Fix a probability $p=p(1)\in(0,1)$. Let $\ttt=(1/p,R,\ldots,R)$ be the $n$-tuple whose first entry is $1/p$ and whose other entries are all equal to $R$. 
According to \cref{thm:main_iTASEP}, the stationary probability of a state $(w,\sigma)$ of $\iTASEP_\lambda$ is \[\pi^{\ii}(w,\sigma)=\frac{1}{n(n-1)}\lim_{R\to\infty}\frac{\Psi_w(\sigma\ttt)}{\Pi_\lambda(\ttt)}.\] 

Consider the following random combinatorial billiard trajectory. Start at a point $z_0$ in the interior of the alcove $\BB$, and shine a beam of light in the direction of the vector ${\ddd^{(n)}=-ne_n+\sum_{j\in[n]}e_j}$. If at some point in time the beam of light is traveling in an alcove $\BB u$ and hits the hyperplane $\HH^{(u,s_i)}$, then it passes through with probability $p\,\hh_0(\overline{u}^{-1}(i),\overline{u}^{-1}(i+1))a_{\overline{u}^{-1}(i)}$, and it reflects with probability $1-p\,\hh_0(\overline{u}^{-1}(i),\overline{u}^{-1}(i+1))a_{\overline{u}^{-1}(i)}$. We can discretize this billiard trajectory as follows: if the beam of light is in the alcove $\BB u$ and it is headed toward the facet of $\BB u$ contained in $\HH^{(u,s_i)}$, then we record the pair $(u,i)\in\affS_n\times\ZZ/n\ZZ$. Applying the quotient map $\affS_n\times\ZZ/n\ZZ\to\SSS_n\times\ZZ/n\ZZ$ defined by $(u,i)\mapsto(\overline u,i)$, we obtain a Markov chain ${\bf M}^{\ii}_{{\ddd^{(n)}}}$ with state space $\SSS_n\times\ZZ/n\ZZ$, which models a certain random combinatorial billiard trajectory in the $(n-1)$-dimensional torus $\mathbb T=V^*/Q^\vee$. This random toric combinatorial billiard trajectory is similar to the Markov chain ${\bf M}_{\ddd^{(n)}}$ discussed in \cref{subsec:intro_billiards,subsec:reducedA}, except now the probability that the beam of light passes through a toric hyperplane $\qqq(\HH_{i,j}^k)$ (with $1\leq i<j\leq n$) depends on both the side of $\qqq(\HH_{i,j}^k)$ that it hits and the parameter $a_j$. Let $\zeta_{{\ddd^{(n)}}}^{\ii}$ denote the stationary distribution of ${\bf M}_{{\ddd^{(n)}}}^{\ii}$. 

In \cref{subsec:billiards1}, we saw that the stationary distribution $\zeta_{\ddd^{(n)}}$ of ${\bf M}_{\ddd^{(n)}}$ is determined by ASEP polynomials. A very similar argument (which we omit) shows that $\zeta_{{\ddd^{(n)}}}^{\ii}$ is determined by inhomogeneous TASEP polynomials. More precisely, \[\zeta_{{\ddd^{(n)}}}^{\ii}(w,M)=\frac{1}{n}\lim_{R\to\infty}\frac{\Psi_w(\ttt^{(M)})}{\Pi_\lambda(\ttt)},\] where $\ttt^{(M)}$ is the $n$-tuple whose $M$-th entry is $1/p$ and whose other entries are all equal to $R$.   

\section{The Stoned Multispecies Open Boundary ASEP}\label{sec:open}

\subsection{Set-up}

In this section, we shift gears and assume that the root system $\Phi$ is of type $C_n$. Thus, $\Phi=\Phi^+\sqcup\Phi^-$, where 
\begin{equation}\label{eq:PhiC}
\Phi^+=\{e_i\pm e_j:1\leq i<j\leq n\}\sqcup\{2e_i:1\leq i\leq n\}\quad\text{and}\quad\Phi^-=-\Phi^+.
\end{equation} 
Here, we have $V=V^*=\mathbb R^n$. The index set $I$ is $[n]$. For $i\in[n-1]$, the simple root $\alpha_i$ is $e_i-e_{i+1}$. Also, $\alpha_n=2e_n$. The coroot lattice is $Q^\vee=\mathbb Z^n$. We identify the Weyl group $W$ with the hyperoctahedral group $C_n$, whose elements are permutations $w$ of the set ${\pm[n]\coloneq\{-n,\ldots,-1,1,\ldots,n\}}$ such that $w(-i)=-w(i)$ for all $i$. Under this identification, each simple reflection $s_i=\overline s_i$ for $i\in[n-1]$ is the element that simultaneously swaps $i$ and $i+1$ and swaps $-i$ and $-(i+1)$. The simple reflection $s_n=\overline s_n$ swaps $n$ and $-n$. Also, $\overline s_0$ swaps $1$ and $-1$. 

Let us fix real parameters $\alpha,\beta\in(0,1)$ in addition to our usual fixed parameter $t\in[0,1)$ and nondecreasing $n$-tuple $\lambda$. Let $S_\lambda^\pm$ denote the set of tuples of the form $(\epsilon_1\mu_1,\ldots,\epsilon_n\mu_n)$, where ${(\mu_1,\ldots,\mu_n)\in S_\lambda}$ and $\epsilon_1,\ldots,\epsilon_n\in\{\pm 1\}$. Suppose $w\in C_n$. For $\mu\in S_\lambda^\pm$, let us write ${w\mu=(\mu_{w^{-1}(1)},\ldots,\mu_{w^{-1}(n)})}$, where $\mu_{-i}=-\mu_i$. Given a tuple $\boldsymbol{y}$ of variables or real numbers, we let $w\boldsymbol{y}=(y_{w^{-1}(1)},\ldots,y_{w^{-1}(n)})$, where $y_{-i}=y_i^{-1}$. (It should be clear from context whether we want to consider a tuple as an element of $S_\lambda^\pm$ or as a tuple of variables or real numbers.)

Let $\mathbb C(\alpha,\beta,t)[{\bf x}^{\pm 1}]$ be the set of Laurent polynomials in the variables $x_1,\ldots,x_n$ whose coefficients are rational functions in $\alpha,\beta,t$. Given a Laurent polynomial ${f=f({\bf x})=f(x_1,\ldots,x_n)}$, we define a Laurent polynomial $wf$ by letting \[wf({\bf x})=f(w{\bf x})=f(x_{w^{-1}(1)},\ldots,x_{w^{-1}(n)}).\]    

Define operators $T_1,\ldots,T_{n-1}$ on $\mathbb C(\alpha,\beta,t)[{\bf x}^{\pm 1}]$ by
\[T_i=t-\frac{tx_i-x_{i+1}}{x_i-x_{i+1}}(1-\overline s_i).\] Let us also define the operators 
\[T_0=1-\frac{1}{\alpha}\frac{(1-x_1^2)\alpha-(1-t)x_1}{1-x_1^2}(1-\overline s_0)\quad\text{and}\quad T_n=1-\frac{1}{\beta}\frac{(1-x_n^2)\beta+(1-t)x_n}{1-x_n^2}(1-\overline s_n).\]
One can check that the following relations hold: 
\begin{alignat*}{2}
T_0^2&=T_n^2=1; && \\
T_0T_1T_0T_1&=T_1T_0T_1T_0; && \\ 
T_nT_{n-1}T_nT_{n-1}&=T_{n-1}T_nT_{n-1}T_n; && \\ 
(T_i-t)(T_i+1)&=0\qquad &&\text{for }1\leq i\leq n-1; \\ 
T_iT_{i+1}T_i&=T_{i+1}T_iT_{i+1} \qquad &&\text{for }1\leq i\leq n-2; \\ 
T_iT_j&=T_jT_i \qquad &&\text{for }0\leq i,j\leq n\text{ with }|i-j|\geq 2. 
\end{alignat*}
(This means that these operators define an action of the \emph{Hecke algebra}\footnote{This Hecke algebra is different from the $0$-Hecke algebra defined in \cref{sec:preliminaries}.} of $\widetilde C_n$.) 

There is a unique family $(G_\mu)_{\mu\in S_\lambda^{\pm}}$ of Laurent polynomials in $\mathbb C(\alpha,\beta,t)[{\bf x}^{\pm 1}]$ such that the coefficient of $x_1^{\lambda_n}x_2^{\lambda_{n-1}}\cdots x_n^{\lambda_1}$ in $G_\lambda$ is $1$ and such that the following \dfn{exchange equations} hold: 
\begin{alignat}{2}
\label{CKZ1} T_0G_\mu&=G_{\overline s_0\mu}; && \\
\label{CKZ2} T_iG_\mu&=G_{\overline{s}_i\mu}\quad&&\text{if }1\leq i\leq n-1\text{ and }\mu_i<\mu_{i+1}; \\ 
\label{CKZ3} \overline s_iG_\mu&=G_{\mu}\quad&&\text{if }1\leq i\leq n-1\text{ and }\mu_i=\mu_{i+1}; \\ 
\label{CKZ4} T_nG_\mu&=G_{\overline{s}_n\mu}. &&  
\end{alignat}
(These equations are obtained by setting $q=1$ in the \emph{$q$KZ equations} from \cite{CGdGW,CMW2,Kasatani}.) 
These Laurent polynomials $G_\mu=G_\mu({\bf x};t)$ are called \dfn{open boundary ASEP polynomials}.\footnote{Our polynomials are actually special instances of the more general polynomials defined in those articles because we set $q=1$.}  The Laurent polynomial \[K_\lambda({\bf x};t)=\sum_{\mu\in S_\lambda}G_\mu({\bf x};t)\] is a \dfn{Koornwinder polynomial} \cite{CGdGW, CMW2}. We remark that $K_\lambda({\bf x};t)$ is symmetric in the variables $x_1,\ldots,x_n$. When all parts of $\lambda$ belong to $\{0,1\}$, Corteel, Mandelshtam, and Williams \cite{CMW2} gave a combinatorial formula for computing open boundary ASEP polynomials using \emph{rhombic staircase tableaux}. 

We let $\CmASEP_\lambda$ denote the \dfn{multispecies open boundary ASEP} whose state space is $S_\lambda^{\pm}$. This is a discrete-time Markov chain in which the transition probability from a state $\mu$ to a state $\mu'$ is 
\[\mathbb P(\mu\to\mu')=\begin{cases} \frac{1}{n+1}\hh_t(\mu_i,\mu_{i+1}) & \mbox{if }1\leq i\leq n-1\text{ and }\mu'=\overline{s}_i\mu\neq\mu; \\ \frac{1}{n+1}\alpha & \mbox{if }\mu'=\overline{s}_0\mu\neq\mu; \\ \frac{1}{n+1}\beta & \mbox{if }\mu'=\overline{s}_n\mu\neq\mu; \\  1-\sum_{\nu\in S_\lambda^\pm\setminus\{\mu\}}\mathbb P(\mu\to\nu) & \mbox{if }\mu=\mu'; \\ 0 & \mbox{otherwise}.  \end{cases}
\]
The state $\mu$ can be visualized as a configuration of particles on a line with sites $1,\ldots,n$ (listed from left to right), where the particle on site $i$ has \dfn{species} $\mu_i$ (see \cref{fig:E}). The stationary distribution of $\CmASEP_\lambda$ has been a topic of vigorous investigation \cite{Cantini2, CFRV, CMRV, CGdGW, CMWAdv, CMW2, CW07, CW18, Uch2, Uch}; in particular, Cantini, Garbali, de Gier, and Wheeler \cite{CGdGW} showed that the stationary probability of a state $\mu\in S_\lambda^{\pm}$ is \[\frac{G_\mu(1,\ldots,1;t)}{K_\lambda(1,\ldots,1;t)}.\]   

\begin{figure}[ht]
\begin{center}{\includegraphics[height=4.160cm]{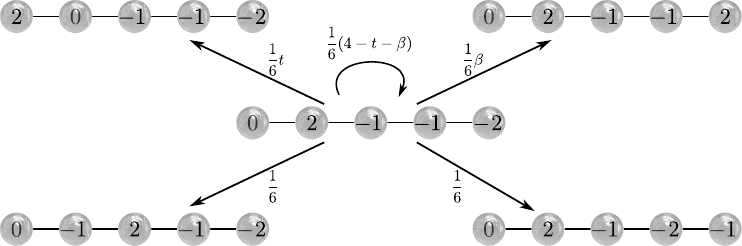}}
\end{center}
\caption{All the transitions from the state $(0,2,-1,-1,-2)$ in $\CmASEP_{(0,1,1,2,2)}$. }
\label{fig:E}
\end{figure}  

Recall the notation $\pm[n]=\{-n,\ldots,-1,1,\ldots,n\}$. We define a bijection ${\iota\colon\pm[n]\to\ZZ/2n\ZZ}$ by 
\begin{equation}\label{eq:iota}
\iota(h)=\begin{cases} h & \mbox{if }1\leq h\leq n; \\ h+1 & \mbox{if }-n\leq h\leq -1.\end{cases}
\end{equation} 
In order to define a stoned variant of $\CmASEP_\lambda$, we need a version of the auxiliary TASEP. We will use a very simple deterministic process that we call the \dfn{auxiliary cyclic shift}. Formally, this system is just a bijection $\aleph\colon\pm[n]\to\pm[n]$ defined by 
\[\aleph(h)=\iota^{-1}(\iota(h)+1).\] To represent a number $h\in\pm[n]$, we place a single stone $\s$ on the site $|h|$ on the line; the stone is \dfn{upright} if $h>0$, and it is \dfn{overturned} if $h<0$. In the auxiliary cyclic shift, the upright stone moves one space to the right at each step until it reaches the right endpoint, where it becomes overturned; the overturned stone then moves one space to the left at each step until it reaches the left endpoint, where it becomes upright again. See \cref{fig:F}.

\begin{figure}[ht]
\begin{center}{\includegraphics[height=3.168cm]{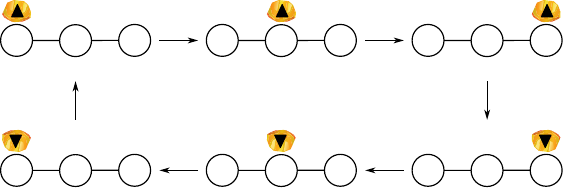}}
\end{center}
\caption{The dynamics of the auxiliary cyclic shift for $n=3$.}
\label{fig:F}
\end{figure}  

Let $\ttt=(\chi_1,\ldots,\chi_{n-1},\chi_{\s})$ be an $n$-tuple of nonzero real numbers. Let 
\begin{equation}\label{eq:p0pn}p_0=\frac{1-\chi_{\s}^2}{(1-t)\chi_{\s}+\alpha(1-\chi_{\s}^2)}\quad\text{and}\quad p_n=\frac{1-\chi_{\s}^2}{(1-t)\chi_{\s}+\beta(1-\chi_{\s}^2)},
\end{equation} and for $1\leq i\leq n-1$, let 
\begin{equation}\label{eq:piupdown}
p^{\uparrow}_i=\frac{\chi_{\s}-\chi_i}{t\chi_{\s}-\chi_i}\quad\text{and}\quad p^{\downarrow}_i=\frac{\chi_{\s}-\chi_i^{-1}}{t\chi_{\s}-\chi_i^{-1}}.
\end{equation} 
We assume that $\chi_1,\ldots,\chi_{n-1},\chi_{\s}$ are chosen so that $p_0,p_n,p^\uparrow_1,\ldots,p^\uparrow_{n-1},p^\downarrow_{1},\ldots,p^\downarrow_{n-1}$ all belong to $(0,1)$. We envision an element of $S_\lambda^\pm\times\pm[n]$ as a configuration of particles on the line with sites $1,\ldots,n$ together with the stone (either upright or overturned) on one of the sites. 
 
\begin{definition}\label{def:stonedobmASEP}
The \dfn{stoned multispecies open boundary ASEP}, which we denote by $\blacktriangle\CmASEP_\lambda$, is the discrete-time Markov chain with state space $S_\lambda\times\pm[n]$ whose transition probabilities are as follows: 
\begin{itemize}
\item If $\mu_1\neq 0$, then $\mathbb P((\mu,-1)\to(\overline{s}_0\mu,1))=p_0\alpha$ and $\mathbb P((\mu,-1)\to(\mu,1))=1-p_0\alpha$. If $\mu_1=0$, then $\mathbb P((\mu,-1)\to(\mu,1))=1$. 
\item If $\mu_n\neq 0$, then $\mathbb P((\mu,n)\to(\overline{s}_n\mu,-n))=p_n\beta$ and $\mathbb P((\mu,n)\to(\mu,-n))=1-p_n\beta$. If $\mu_n=0$, then $\mathbb P((\mu,n)\to(\mu,-n))=1$. 
\item For $1\leq i\leq n-1$, we have \[\mathbb P((\mu,i)\to(\overline{s}_i\mu,i+1))=p_i^\uparrow\hh_t(\mu_i,\mu_{i+1})\quad\text{if }\mu\neq \overline s_i\mu,\] and we have \[\mathbb P((\mu,i)\to(\mu,i+1))=1-p_i^\uparrow\hh_t(\mu_i,\mu_{i+1}).\] 
\item For $1\leq i\leq n-1$, we have \[\mathbb P((\mu,-(i+1))\to(\overline{s}_i\mu,-i))=p_i^\downarrow\hh_t(\mu_i,\mu_{i+1})\quad\text{if }\mu\neq\overline s_i\mu,\] and we have \[\mathbb P((\mu,-(i+1))\to(\mu,-i))=1-p_i^\downarrow\hh_t(\mu_i,\mu_{i+1}).\]  
\item All other transition probabilities are $0$. 
\end{itemize}
\end{definition}

\cref{fig:G} illustrates \cref{def:stonedobmASEP}. 

\begin{figure}[ht]
\begin{center}{\includegraphics[height=10.96cm]{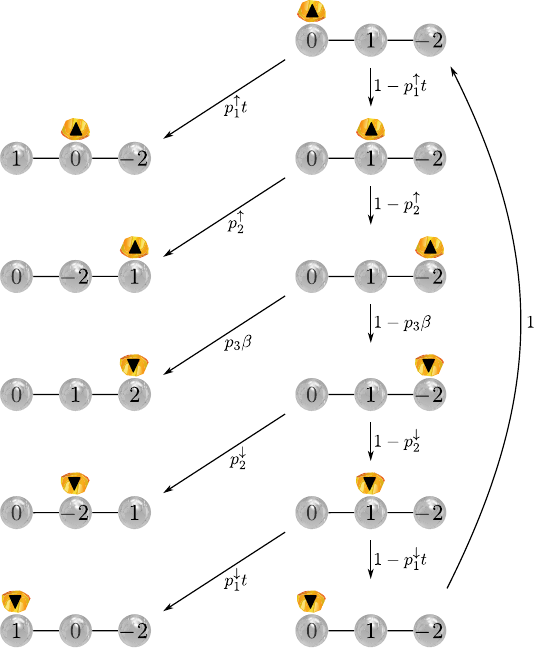}}
\end{center}
\caption{Some of the transitions in $\CmASEP_{(0,1,2)}$.}
\label{fig:G}
\end{figure}  

There is a simple intuitive description of the dynamics of $\blacktriangle\CmASEP_\lambda$. The stone moves deterministically according to the auxiliary cyclic shift. \begin{itemize}
\item When the stone changes from being overturned to being upright (respectively, from being upright to being overturned), it sends a signal to the particle on site $1$ (respectively, site $n$). The signal reaches the particle with probability $p_0$ (respectively, $p_n$). If the particle receives the signal, then with probability $\alpha$ (respectively, $\beta$), it decides to follow orders and change the sign of its species. If the particle does not receive the signal or it decides not to follow orders, then it does nothing. 
\item When the upright (respectively, overturned) stone moves from site $i$ to site $i+1$ (respectively, from site $i+1$ to site $i$), it sends a signal to the particles on sites $i$ and $i+1$ telling them to swap places. Let $\mu_i$ and $\mu_{i+1}$ be the species of the particles on sites $i$ and $i+1$. The particles receive the signal with probability $p_i^\uparrow$ (respectively, $p_i^\downarrow$), and if they receive the signal, then they decide to follow orders and swap places with probability $\hh_t(\mu_i,\mu_{i+1})$. If the particles do not receive the signal or they decide not to follow orders, then they do not move.    
\end{itemize} 

\begin{remark}\label{rem:abcd}
There is a more general version of the multispecies open boundary ASEP in which the two boundary rates $\alpha,\beta$ are replaced by four boundary rates $\alpha,\beta,\gamma,\delta$ (see, e.g., \cite{Cantini2, CGdGW, CMW2}). In this setting, the probability $\mathbb P(\mu\to\overline s_0\mu)$ is $\frac{1}{n+1}\alpha$ if $\mu_1<0$ and is $\frac{1}{n+1}\gamma$ if $\mu_1>0$, while the probability $\mathbb P(\mu\to\overline s_n\mu)$ is $\frac{1}{n+1}\beta$ if $\mu_n>0$ and is $\frac{1}{n+1}\delta$ if $\mu_n<0$. Our derivation of the stationary distribution of the stoned multispecies open boundary ASEP in \cref{thm:main_obASEP} requires us to assume that $\alpha=\gamma$ and $\beta=\delta$. It would be interesting to extend our results to the setting where $\alpha,\beta,\gamma,\delta$ can be distinct. That said, it \emph{is} possible to generalize the results in this section by replacing the auxiliary cyclic shift with an \emph{auxiliary open boundary TASEP} using $n$ stones of various densities, just as we did in \cref{sec:ring}. In this more general setting, the auxiliary TASEP is essentially a copy of a multispecies open boundary TASEP with $\gamma=\delta=0$, and its stationary distribution can be computed using the open boundary ASEP polynomials from \cite{CGdGW,CMW2}. In order to keep our presentation as simple as possible, we have opted to work in the setting where there is only one stone that moves deterministically. This restricted setting is still general enough to exhibit the behavior that we find most interesting.  
\end{remark} 

\subsection{Stationary Distribution} 

The assumption that the probabilities in \eqref{eq:p0pn} and \eqref{eq:piupdown} are strictly between $0$ and $1$ guarantees that $\blacktriangle\CmASEP_\lambda$ is irreducible; we will determine its stationary distribution in terms of open boundary ASEP polynomials. First, we need the following lemma; we omit the proof because it is virtually identical to that of \cref{lem:rewriting}.  

\begin{lemma}\label{lem:rewritingC}
For $\nu\in S_\lambda^\pm$ and $1\leq i\leq n-2$, we have 
\[G_\nu({\bf x};t)=\left(1-\frac{x_{i+1}-x_i}{tx_{i+1}-x_i}\hh_t(\nu_{i},\nu_{i+1})\right)\overline{s}_iG_\nu({\bf x};t)+\frac{x_{i+1}-x_i}{tx_{i+1}-x_i}\hh_t(\nu_{i+1},\nu_{i})\overline{s}_iG_{\overline{s}_i\nu}({\bf x};t).\]
\end{lemma}

For $k\in[n]$, let 
\[\ttt^{(k)}=(\chi_1,\ldots,\chi_{k-1},\chi_{\s},\chi_k,\ldots,\chi_{n-1})\quad\text{and}\quad\ttt^{(-k)}=(\chi_1,\ldots,\chi_{k-1},\chi_{\s}^{-1},\chi_k,\ldots,\chi_{n-1}).\] 

\begin{theorem}\label{thm:main_obASEP}
The stationary probability measure $\pi^{\ob}$ of $\blacktriangle\CmASEP_\lambda$ is given by \[\pi^{\ob}(\mu,h)=\frac{1}{2n}\frac{G_{\mu}(\ttt^{(h)};t)}{K_\lambda(\ttt;t)}.\]
\end{theorem}

\begin{proof}
Fix a state $(\mu,h)\in S_\lambda^\pm\times\pm[n]$. Let $h'=\aleph^{-1}(h)$. We will prove that the desired stationary distribution satisfies the Markov chain balance equation at $(\mu,h)$. This amounts to showing that 
\begin{equation}\label{eq:balance_C}\sum_{(\widehat\mu,\widehat h)\in S_\lambda^\pm\times\pm[n]}\mathbb P((\widehat\mu,\widehat h)\to(\mu,h))G_{\widehat \mu}(\ttt^{(\widehat h)};t)=G_\mu(\ttt^{(h)};t).
\end{equation} 
Note that the sum on the left-hand side of this equation has only two nonzero terms. We consider four cases. 

\medskip 

\noindent {\bf Case 1.} Assume $h=1$. In this case, $h'=-1$, and the balance equation \eqref{eq:balance_C} becomes  
\[\mathbb P((\overline{s}_{0}\mu,-1)\to(\mu,1))G_{\overline{s}_{0}\mu}(\ttt^{(-1)};t)+\mathbb P((\mu,-1)\to(\mu,1))G_{\mu}(\ttt^{(-1)};t)=G_\mu(\ttt^{(1)};t).\] 
Appealing to \cref{def:stonedobmASEP} and noting that $\ttt^{(-1)}=\overline s_0\ttt^{(1)}$, we can rewrite this desired equation as 
\begin{equation}\label{eq:desired1}
p_0\alpha \overline s_0G_{\overline s_0\mu}(\ttt^{(1)};t)+(1-p_0\alpha)\overline s_0G_\mu(\ttt^{(1)};t)=G_\mu(\ttt^{(1)};t).
\end{equation}
Noting that 
\[\frac{1}{\alpha}\frac{(1-\chi_{\s}^{-2})\alpha-(1-t)\chi_\s^{-1}}{1-\chi_{\s}^{-2}}=\frac{1}{p_0\alpha},\] we can apply \eqref{CKZ1} with ${\bf x}=\overline{s}_0\ttt^{(1)}$ to find that 
\[\left(1-\frac{1}{p_0\alpha}\right)\overline{s}_0G_\mu(\ttt^{(1)};t)+\frac{1}{p_0\alpha}G_\mu(\ttt^{(1)};t)=\overline{s}_0G_{\overline{s}_0\mu}(\ttt^{(1)};t),\] and this is equivalent to \eqref{eq:desired1}. 

\medskip 

\noindent {\bf Case 2.} Assume $h=-n$. Then $h'=n$. The proof in this case is very similar to the proof in Case~1, so we omit it. 

\medskip 

\noindent {\bf Case 3.} Assume $2\leq h\leq n$. In this case, $h'=h-1$, and the balance equation \eqref{eq:balance_C} becomes 
\[
\mathbb P((\overline{s}_{h-1}\mu,h-1)\to(\mu,h))G_{\overline{s}_{h-1}\mu}(\ttt^{(h-1)};t)+\mathbb P((\mu,h-1)\to(\mu,h))G_\mu(\ttt^{(h-1)};t)=G_\mu(\ttt^{(h)};t).
\]
Appealing to \cref{def:stonedobmASEP} and setting $i=h-1$, we can write this equation as 
\[p_{i}^\uparrow\hh_t(\mu_{i+1},\mu_i)G_{\overline{s}_{i}\mu}(\ttt^{(i)};t)+(1-p^\uparrow\hh_t(\mu_i,\mu_{i+1}))G_\mu(\ttt^{(i)};t)=G_\mu(\ttt^{(i+1)};t).\] This follows immediately from setting $\nu=\mu$ and ${\bf x}=\ttt^{(i+1)}$ in \cref{lem:rewritingC} and noting (by \eqref{eq:piupdown}) that 
\[p_i^\uparrow=\frac{\chi_\s-\chi_i}{t\chi_\s-\chi_i}=\frac{x_{i+1}-x_i}{tx_{i+1}-x_i}.\]

\medskip

\noindent {\bf Case 4.} Assume $-(n-1)\leq h\leq -1$. In this case, $h'=h-1$, and the balance equation \eqref{eq:balance_C} becomes 
\[
\mathbb P((\overline{s}_{-h}\mu,h-1)\to(\mu,h))G_{\overline{s}_{-h}\mu}(\ttt^{(h-1)};t)+\mathbb P((\mu,h-1)\to(\mu,h))G_\mu(\ttt^{(h-1)};t)=G_\mu(\ttt^{(h)};t).
\] 
Appealing to \cref{def:stonedobmASEP} and setting $i=-h$, we can write this equation as 
\[p_i^\downarrow\hh_t(\mu_{i+1},\mu_i)G_{\overline{s}_{i}\mu}(\ttt^{(h-1)};t)+(1-p_i^\downarrow\hh_t(\mu_i,\mu_{i+1}))G_\mu(\ttt^{(h-1)};t)=G_\mu(\ttt^{(h)};t).\] This follows immediately from setting $\nu=\mu$ and ${\bf x}=\ttt^{(h)}$ in \cref{lem:rewritingC} and noting (by \eqref{eq:piupdown}) that 
\[p_i^\downarrow=\frac{\chi_\s-\chi_i^{-1}}{t\chi_\s-\chi_i^{-1}}=\frac{\chi_i-\chi_{\s}^{-1}}{t\chi_i-\chi_\s^{-1}}=\frac{x_{i+1}-x_i}{tx_{i+1}-x_i}. \qedhere\]
\end{proof}

\begin{remark}\label{rem:close_to_0_C} 
If the probabilities in \eqref{eq:p0pn} and \eqref{eq:piupdown} are all very close to $0$, then $\chi_1,\ldots,\chi_{n-1},\chi_{\s}$ are very close to $1$, so $2n\pi^{\ob}(\mu,i)$ is very close to the stationary probability of $\mu$ in $\CmASEP_\lambda$. As in \cref{rem:close_to_0,rem:close_to_0_i}, this is because the signals that the stone sends to the particles rarely reach the particles. 
\end{remark} 

\subsection{Random Combinatorial Billiards}\label{subsec:billiards3} 

Assume in this subsection that $\lambda=(1,2,\ldots,n)$, and use the map $w\mapsto(w^{-1}(1),\ldots,w^{-1}(n))$ to identify $C_n$ with $S_\lambda^\pm$. 

Recall that the hyperplanes in $\mathcal H_{\widetilde C_n}$ are of the form $\HH_{\beta}^k$ for $\beta\in\Phi^+$ and $k\in\ZZ$, where $\Phi^+$ is the root system of type~$C_n$ given in \eqref{eq:PhiC}. Let 
\[\mathcal H_{\widetilde C_n}^{0}=\{\HH_{2e_i}^{k}:i\in[n],\,k+1\in 2\ZZ\} \quad\text{and}\quad \mathcal H_{\widetilde C_n}^{n}=\{\HH_{2e_i}^{k}:i\in[n],\,k\in 2\ZZ\},\] and let \[\mathcal H_{\widetilde C_n}^{[n-1]}=\{\HH_{e_i\pm e_j}^k:1\leq i<j\leq n,\, k\in\ZZ\}\] so that \[\mathcal H_{\widetilde C_n}=\mathcal H_{\widetilde C_n}^{0}\sqcup \mathcal H_{\widetilde C_n}^{n}\sqcup \mathcal H_{\widetilde C_n}^{[n-1]}.\] It is known that every hyperplane of the form $\HH^{(u,s_0)}$ is in $\mathcal H_{\widetilde C_n}^0$, every hyperplane of the form $\HH^{(u,s_n)}$ is in $\mathcal H_{\widetilde C_n}^n$, and every hyperplane of the form $\HH^{(u,s_i)}$ with $i\in[n-1]$ is in $\mathcal H_{\widetilde C_n}^{[n-1]}$. 

Let \[p=\frac{\chi_\s-1}{t\chi_\s-1},\] and assume that $\chi_1=\cdots=\chi_{n-1}=1$ so that $p_i^\uparrow=p_i^\downarrow=p$ for all $1\leq i\leq n-1$. For $u\in\widetilde C_n$ and $i\in\widetilde I=\{0,1,\ldots,n\}$, let \[{\bf p}(u,s_i)=\begin{cases} p_0\alpha & \mbox{if }i=0; \\  p_n\beta & \mbox{if }i=n; \\ p\hh_t(\overline{u}^{-1}(i),\overline{u}^{-1}(i+1)) & \mbox{if }i\in[n-1].  \end{cases}\] 

Consider the following random billiard trajectory. Start at a point $z_0$ in the interior of $\BB$, and shine a beam of light in the direction of the vector $e_1$. If at some point in time the beam of light is traveling in an alcove $\BB u$ and it hits the hyperplane $\HH^{(u,s_i)}$, then it passes through with probability ${\bf p}(u,s_i)$, and it reflects with probability $1-{\bf p}(u,s_i)$. Note that the probability of the light beam passing through the hyperplane only depends on which of the three sets $\mathcal H_{\widetilde C_n}^0,\mathcal H_{\widetilde C_n}^n,\mathcal H_{\widetilde C_n}^{[n-1]}$ contains the hyperplane and which side of the hyperplane the light beam hits. 

The vector $e_1\in Q^\vee$ belongs to the set $\Upsilon_{z_0}$ (defined in \cref{subsec:intro_billiards}), and its corresponding word is $\mathsf{w}(e_1)=\cdots s_{i_1}s_{i_0}=\cdots\mathsf{vvv}$, where $\mathsf{v}=s_{1}\cdots s_{n-1}s_ns_{n-1}\cdots s_1s_0$. Thus, the period of $\mathsf{w}(e_1)$ is $N_{e_1}=2n$, and we have $i_j=i_{j+2n}$ for all $j\geq 0$. 

We can discretize the billiard trajectory described above as follows. Let $u_M$ be the alcove containing the beam of light after it hits a hyperplane in $\mathcal H_{\widetilde C_n}$ for the $M$-th time; at this point in time, the beam of light is facing toward the facet of $\BB u_M$ contained in $\HH^{(u_M,s_{i_M})}$. This yields a discrete-time Markov chain whose state at time $M$ is the pair $(u_M,M)$ in $\widetilde C_n\times \ZZ/2n\ZZ$. By projecting this Markov chain through the quotient map ${\widetilde C_n\times \ZZ/2n\ZZ\to C_n\times \ZZ/2n\ZZ}$ defined by $(u,i)\mapsto(\overline{u},i)$, we obtain a Markov chain ${\bf M}_{e_1}^{(\alpha,\beta,t)}$ on $C_n\times\ZZ/2n\ZZ$, which models a certain random combinatorial billiard trajectory in the $n$-dimensional torus ${\mathbb T=V^*/Q^\vee=\mathbb R^n/\mathbb Z^n}$ (see \cref{fig:H}). This Markov chain is isomorphic to $\blacktriangle\CmASEP_\lambda$; the isomorphism is just the map $C_n\times\ZZ/2n\ZZ\to C_n\times\pm[n]$ given by $(w,i)\mapsto(w,\iota^{-1}(i))$, where $\iota\colon\pm[n]\to\ZZ/2n\ZZ$ is the bijection defined in \eqref{eq:iota}. Hence, it follows from \cref{thm:main_obASEP} that the stationary probability of a state $(w,i)$ in ${\bf M}_{e_1}^{(\alpha,\beta,t)}$ is \[\frac{1}{2n}\frac{G_w(\ttt^{(\iota^{-1}(i))};t)}{K_{\lambda}(\ttt;t)}.\] 

\begin{figure}[ht]
\begin{center}{\includegraphics[width=0.495\linewidth]{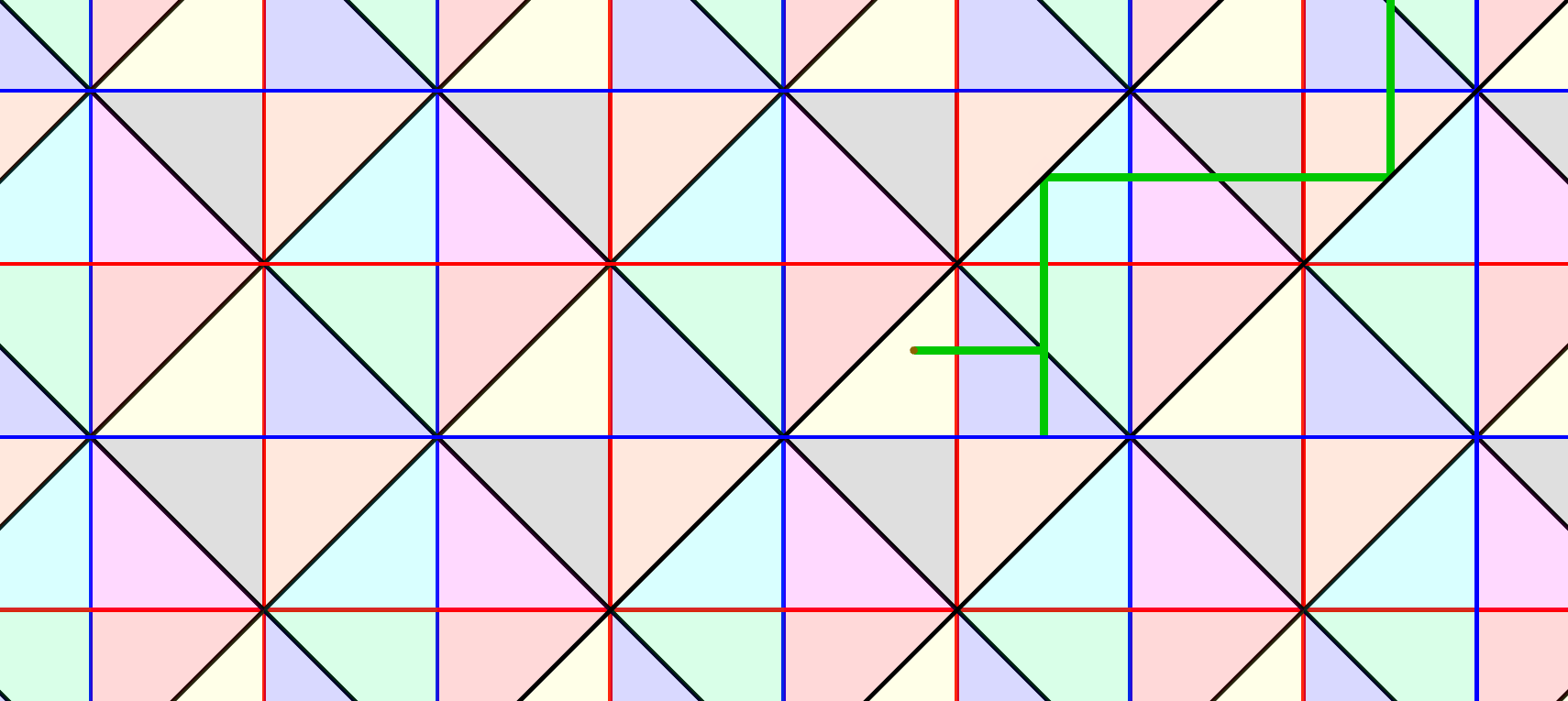} \\ \vspace{0.4cm}\includegraphics[width=\linewidth]{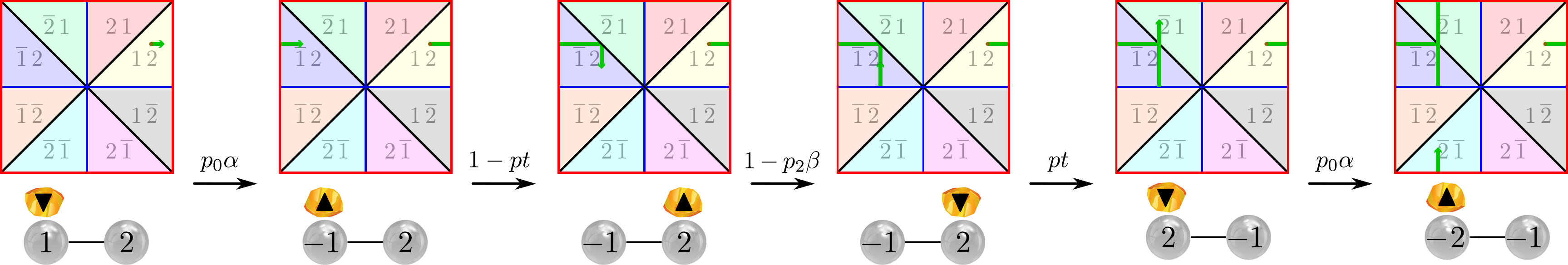}}
\end{center}
\caption{On the top is a ({\color{MildGreen}green}) random billiard trajectory in $\widetilde C_2$. Hyperplanes in $\mathcal H_{\widetilde C_2}^0$, $\mathcal H_{\widetilde C_n}^2$, and $\mathcal H_{\widetilde C_n}^{[1]}$ are in {\color{red}red}, {\color{blue}blue}, and black, respectively. The bottom image shows the first five transitions in ${\bf M}_{e_1}^{(\alpha,\beta,t)}$ corresponding to the beginning of the billiard trajectory from the top image. Each transition is labeled with its probability. Each state is represented both as a beam of light traveling in the $2$-dimensional torus $\mathbb T$ and as a configuration of two particles and a stone on a line. The state with a particle of species $i$ on the left and a particle of species $j$ on the right corresponds to the element $w\in C_2$ satisfying $w(i)=1$ and $w(j)=2$; this state corresponds to the toric alcove labeled by the word {\color{gray}$w(1)w(2)$} (with overlines representing negative numbers). }   
\label{fig:H}
\end{figure}  

\section{Other Directions}\label{sec:other}

Let us now discuss some other potential directions for future work. 

\subsection{Other Weyl Groups and Directions} 
Our initial other direction is about other initial directions. Most of our work focused on combinatorial billiard trajectories in $\affS_n$ (or $\widehat\SSS_n$) in which the beam of light initially travels in the direction of the coroot vector $\ddd^{(n)}=-ne_n+\sum_{j\in[n]}e_j$. It would be interesting to see if some of our results from this case extend to the setting where we replace the initial direction vector $\ddd^{(n)}$ by a different nonzero vector. What happens if one chooses an initial direction randomly? It is also very natural to see how much of our work extends to other affine Weyl groups (or possibly even other Coxeter groups). 

\subsection{Breaking Symmetry} 
In \cref{thm:LamBilliards2}, we had to initiate the random process by choosing $\epsilon\in\{\pm 1\}$ uniformly at random and shining the beam of light in the direction of $\epsilon\dd$. The reason for doing this was to add additional symmetry so that we could prove \eqref{eq:xxinverse}. It would be very interesting to obtain an analogous result when we break this symmetry and always choose $\epsilon=1$. 

\subsection{Convergence Rates}

Lam's \cref{thm:Lam} tells us that the unit vector pointing in the same direction as the state of the reduced random walk almost surely converges to a vector in $\langle\psi_\Lam\rangle W$. One could ask for quantitative bounds on the rate of this convergence; this would likely make use of estimates on the mixing time of the multispecies TASEP (which are notoriously difficult to obtain). Likewise, \cref{thm:LamBilliards} tells us that the unit vector pointing in the same direction as the state of the reduced random combinatorial billiard trajectory $\widetilde{\bf M}_\dd$ almost surely converges to a vector in $\langle\psi_\dd\rangle W$. It would be interesting to obtain quantitative bounds on the rate of this convergence. One could also restrict to the special case where $W=\SSS_n$ and $\dd=\ddd^{(n)}$; here, such estimates for the convergence rate would likely require one to derive estimates on the mixing time of the stoned multispecies TASEP, which could be interesting in its own right. Schmid and Sly \cite{SchmidSly} analyzed the mixing time of the $2$-species TASEP on a ring with $n$ sites; perhaps it could be fruitful to consider the mixing time of the $2$-species stoned TASEP with $1$ stone of density $1$ and $n-1$ stones of density~$2$.  

\subsection{Stoned Exclusion Processes} 
We believe it should be possible to define and study stoned variants of other notable Markov chains beyond the three examples constructed and analyzed in \cref{sec:ring,sec:inhomogeneous,sec:open}. It would be very interesting to develop a general theory of stoned Markov chains. Such a theory would likely concern a ``stoning operation'' that uses one Markov chain (the \emph{auxiliary chain}) to drive the transitions among other objects (playing the roles of the particle configurations). 

There are several other fascinating aspects of the multispecies (T)ASEP and related models, and it is natural to study the analogous aspects of stoned exclusion processes. For example, one could consider statistics such as particle currents, as in \cite{AMW,ANP}. 

Also, recall \cref{rem:abcd}, which asks if it is possible to generalize \cref{thm:main_obASEP} by replacing the two parameters $\alpha,\beta$ with four (possibly all distinct) parameters $\alpha,\beta,\gamma,\delta$. A natural special case that could be worth considering is that in which $\gamma=\delta=0$. 

\subsection{Correlations} 
In \cref{sec:correlations}, we considered the stoned $n$-species TASEP with a single stone of density $1$ and $n-1$ stones of density $2$, and we computed certain $2$-point correlations. Namely, for $1\leq i<j\leq n$, we found a formula for the stationary probability of the event that the particles of species $i$ and $j$ are on sites $1$ and $n$, respectively. This was the only correlation result we needed for our proofs of \cref{thm:AyyerLinussonBilliards,cor:shape,prop:shape}. However, one could also consider other $2$-point correlations or even $3$-point correlations, as Ayyer and Linusson did for the multispecies TASEP in \cite{AyyerLinusson}. For instance, one could ask for the stationary probability of the event that the particles of species $i$ and $j$ are on sites $n$ and $1$, respectively, when $1\leq i<j\leq n$. One could also ask for $2$-point correlations involving sites that are not adjacent.

\subsection{Refractions} 
The article \cite{DefantRefractions} introduces a notion of \emph{refractions} in combinatorial billiards. Roughly speaking, when a beam of light hits a hyperplane, it can \emph{refract} by passing through the hyperplane and then moving in the direction opposite to the direction that it would have moved if it had reflected. This notion of refraction has also been the subject of several works in dynamics over the last several years \cite{Baird,Barutello,Davis1,Davis2,DeBlasi2,
DeBlasi1,Glendinning,Jay,Paris}. 

It would be very interesting to generalize the stochastic models that we have introduced by allowing refractions. For example, suppose we choose parameters $p,p'\in(0,1)$ satisfying ${p+p'\leq 1}$. One could consider a generalization of the reduced random billiard trajectory in which the beam of light has probability $p$ of reflecting, probability $p'$ of refracting, and probability $1-p-p'$ of passing directly through whenever it hits a hyperplane in $\mathcal H_{\widetilde W}$ that it has not already crossed. It seems likely that one could also use the combinatorial formulations of refractions from \cite{DefantRefractions} in order to define generalizations of stoned exclusion processes that (in some cases) encode random combinatorial billiard trajectories with refractions.

\section*{Acknowledgments}
The author was supported by the National Science Foundation under Award No.\ 2201907 and by a Benjamin Peirce Fellowship at Harvard University. He thanks Jimmy He, Noah Kravitz, Olya Mandelshtam, Yelena Mandelshtam, Matthew Nicoletti, and Alan Yan for stimulating discussions. He is especially grateful to Lauren Williams for several extremely helpful conversations and for comments on a preliminary draft of this manuscript. He also thanks the anonymous referees for helpful comments.

\end{document}